\documentclass[12pt]{amsart}

\usepackage{tikz-cd}
\usepackage{amsmath}
\usepackage{amsthm}
\usepackage{hyperref}
\usepackage[margin=1in]{geometry}

\numberwithin{equation}{section}

\newtheorem{prop}{Proposition}
\newtheorem{lemma}[prop]{Lemma}
\newtheorem{thm}[prop]{Theorem}
\newtheorem{cor}[prop]{Corollary}
\numberwithin{prop}{section}

\theoremstyle{definition}
\newtheorem{defn}[prop]{Definition}
\newtheorem{ex}[prop]{Example}
\newtheorem{rmk}[prop]{Remark}

\newcommand{\del}{\partial}
\newcommand{\delb}{\bar{\partial}}\newcommand{\dt}{\frac{\partial}{\partial t}}
\newcommand{\brs}[1]{\left| #1 \right|}

\newcommand{\gG}{\Gamma}
\renewcommand{\gg}{\gamma}
\newcommand{\gD}{\Delta}

\newcommand{\gs}{\sigma}

\newcommand{\gU}{\Upsilon}
\newcommand{\gL}{\Lambda}
\newcommand{\gk}{\kappa}
\newcommand{\gl}{\lambda}

\newcommand{\gw}{\omega}
\newcommand{\ga}{\alpha}
\newcommand{\gb}{\beta}
\renewcommand{\ge}{\epsilon}
\newcommand{\N}{\nabla}
\newcommand{\FF}{\mathcal F}

\newcommand{\EE}{\mathcal E}

\newcommand{\til}[1]{\widetilde{#1}}

\renewcommand{\bar}[1]{\overline{#1}}

\renewcommand{\i}{\sqrt{-1}}
\newcommand{\QQ}{\mathcal Q}

\newcommand{\bgb}{\bar{\beta}}

\newcommand{\bc}{\bar{c}}

\newcommand{\bi}{\bar{i}}
\newcommand{\bj}{\bar{j}}
\newcommand{\bk}{\bar{k}}
\newcommand{\bl}{\bar{l}}
\newcommand{\bm}{\bar{m}}

\newcommand{\bp}{\bar{p}}
\newcommand{\bq}{\bar{q}}
\newcommand{\br}{\bar{r}}
\newcommand{\bs}{\bar{s}}

\newcommand{\IP}[1]{\left<#1\right>}

\DeclareMathOperator{\Rc}{Rc}

\DeclareMathOperator{\tr}{tr}
\DeclareMathOperator{\Ker}{Ker}
\DeclareMathOperator{\Id}{Id}

\DeclareMathOperator{\End}{End}

\begin{document}

\title[Non-K\"ahler Calabi-Yau geometry and pluriclosed flow]{Non-K\"ahler Calabi-Yau geometry and pluriclosed flow}

\begin{abstract} Hermitian, pluriclosed metrics with vanishing Bismut-Ricci form give a natural extension of Calabi-Yau metrics to the setting of complex, non-K\"ahler manifolds, and arise independently in mathematical physics.  We reinterpret this condition in terms of the Hermitian-Einstein equation on an associated holomorphic Courant algebroid, and thus refer to solutions as Bismut Hermitian-Einstein.  This implies Mumford-Takemoto slope stability obstructions, and using these we exhibit infinitely many topologically distinct complex manifolds in every dimension with vanishing first Chern class which do not admit Bismut Hermitian-Einstein metrics.  This reformulation also leads to a new description of pluriclosed flow in terms of Hermitian metrics on holomorphic Courant algebroids, implying new global existence results, in particular on all complex non-K\"ahler surfaces of Kodaira dimension $\gk \geq 0$.  On complex manifolds which admit Bismut-flat metrics we show global existence and convergence of pluriclosed flow to a Bismut-flat metric, which in turn gives a classification of generalized K\"ahler structures on these spaces.
\end{abstract}

\author{Mario Garcia-Fernandez}
\address{Universidad Aut\'onoma de Madrid and
	Instituto de Ciencias Matem\'aticas (CSIC-UAM-UC3M-UCM)\\ Ciudad
	Universitaria de Cantoblanco\\ 28049 Madrid, Spain}
\email{\href{mailto:mario.garcia@icmat.es}{mario.garcia@icmat.es}}

\author{Joshua Jordan}
\address{Rowland Hall\\
         University of California, Irvine\\
         Irvine, CA 92617}
\email{\href{mailto:jpjorda1@uci.edu}{jpjorda1@uci.edu}}

\author{Jeffrey Streets}
\address{Rowland Hall\\
         University of California, Irvine\\
         Irvine, CA 92617}
\email{\href{mailto:jstreets@uci.edu}{jstreets@uci.edu}}

\date{June 17, 2021}

\thanks{
The first author is partially supported by the Spanish Ministry of Science and Innovation, through the `Severo Ochoa Programme for Centres of Excellence in R\&D' (CEX2019-000904-S) and under grants PID2019-109339GA-C32 and EUR2020-112265.  The third author is supported by NSF-1454854.
}

\maketitle

\section{Introduction}

The Calabi-Yau Theorem \cite{YauCC} gives a definitive answer to the questions of existence and uniqueness of K\"ahler, Ricci-flat metrics on a given complex manifold, showing that when $c_1(M) = 0$, there exists a unique such metric in every K\"ahler class.  Later Cao \cite{CaoKRF} proved this theorem using K\"ahler-Ricci flow, showing global existence and convergence to a Ricci-flat metric with arbitrary initial data.  These profound results have applications throughout mathematics and physics, and in recent years there has been interest in extending this theory beyond K\"ahler manifolds to more general settings in complex geometry.  Given a complex manifold $(M^{2n}, J)$ let $g$ denote a Hermitian metric with K\"ahler form $\gw = gJ$.  We say that the metric is pluriclosed if
\begin{align*}
H := - d^c \gw, \qquad d H = 0.
\end{align*}
This is a natural generalization of the K\"ahler condition $d \gw = 0$, and for instance pluriclosed metrics exist on all compact complex surfaces \cite{Gauduchon1form}.  These metrics appear naturally in the index theory of non-K\"ahler manifolds \cite{Bismut}, and in the physics of supersymmetry.  There are many natural Hermitian connections in this setting, and here we focus on the Bismut connection (\cite{Bismut}, \cite{Strominger})
\begin{align*}
\N^B = \N + \tfrac{1}{2} g^{-1} H,
\end{align*}
where $\N$ denotes the Levi-Civita connection of $g$.  This is a Hermitian connection, and as such determines a representative of $c_1(M)$ via
\begin{align*}
\rho_B = \tfrac{1}{2} \tr \Omega^B J \in \Lambda^2(M).
\end{align*}
This is a natural Ricci-type curvature which coincides with the usual Ricci form in case $g$ is K\"ahler.  We say that a pluriclosed metric is \emph{Bismut Hermitian-Einstein} (the reason for this terminology will be explained below) if
\begin{align*}
\rho_B \equiv 0.
\end{align*}
Expressed in terms of the usual Ricci curvature, this equation is equivalent to
\begin{align*}
\Rc - \tfrac{1}{4} H^2 + \tfrac{1}{2} L_{\theta^{\sharp}} g =&\ 0,\\
d^* H - d \theta + i_{\theta^{\sharp}} H =&\ 0,
\end{align*}
where $H^2(X,Y) = \IP{i_X H, i_Y H}$, and $\theta = J d^* \gw$ is the Lee form.  This equation is thus a generalization of the Calabi-Yau condition to complex, non-K\"ahler geometry.  Furthermore, in the compact case it follows that Bismut Hermitian-Einstein metrics automatically satisfy a \emph{generalized Ricci soliton} equation (cf. Proposition \ref{p:solitons})
\begin{align*}
\Rc - \tfrac{1}{4} H^2 + \N^2 f =&\ 0,\\
d^* H + i_{\N f} H =&\ 0.
\end{align*}
These match the equations of motion in compactifications of type II supergravity \cite{Friedanetal}.

Taking inspiration from the Calabi-Yau Theorem, a natural initial question is whether the condition $c_1(M) = 0$ guarantees the existence of a Bismut Hermitian-Einstein metric\footnote{We note that we mean here and throughout vanishing as a de Rham class, as opposed to a Bott-Chern class.  See \cite{TosattiNKCY} for a review of non-K\"ahler Calabi-Yau geometries assuming $c_1 = 0$ in Bott-Chern sense.}. Already in the first nontrivial dimension $n=2$ we see a number of new issues arise, and the answer is emphatically no.  First, it follows from a theorem of Gauduchon-Ivanov \cite{GauduchonIvanov} (cf. \cite[Theorem 8.26]{garciafernStreets2020}) that the only non-K\"ahler solutions occur on \emph{standard Hopf surfaces}, where (up to finite quotients) the metric is isometric to the product metric $g_{S^3} \times g_{S^1}$.  Since all minimal non-K\"ahler complex surfaces of Kodaira dimension $\gk \geq 0$ have vanishing first Chern class, they are thus examples showing that one cannot guarantee existence in general.  Furthermore, there exist pluriclosed metrics which are generalized Ricci solitons on non-standard Hopf surfaces (\cite{Streetssolitons, SU1}), which fit into a smooth family with the example just described.  A complete existence theory should naturally incorporate such solitons.  

Our first step in tackling these subtleties is to reformulate the Bismut Hermitian-Einstein condition using the tools of \emph{generalized geometry}.  Note that every pluriclosed metric determines a class $[\del \gw]$ in the first \v Cech cohomology group $H^{1}(\gL^{2,0}_{cl})$ of the sheaf $\gL^{2,0}_{cl}$ of closed $(2,0)$-forms on $M$.  Such a class corresponds to an isomorphism class of exact holomorphic Courant algebroids \cite{GualtieriGKG}, which are holomorphic vector bundles fitting into an exact sequence
\begin{align*}
0 \longrightarrow T^*_{1,0} \overset{\pi^*}{\longrightarrow} \mathcal{Q} \overset{\pi}{\longrightarrow} T^{1,0} \longrightarrow 0
\end{align*}
and satisfying further axioms (cf. Definition \ref{d:CAhol}).  Furthermore, the structure of the holomorphic vector bundle is given by a twisted $\delb$-operator determined by $\del \gw$ \cite{Bismut}.  Pluriclosed metrics with a fixed torsion class $[\sqrt{-1}\tau] \in H^1(\gL^{2,0}_{cl})$ can be identified with pairs $(\gw, \gb)$, where $\tau - \del \gw = d \gb$, and using this description we determine a generalized Hermitian metric on $\QQ$ via (see Lemma \ref{l:SKTequivalence})
\begin{gather*}
G = \left( 
\begin{matrix}
g_{i \bj} + \gb_{i k} \bgb_{\bj \bl} g^{\bl k} & \i \gb_{ip} g^{\bl p}\\
- \i \bgb_{\bj \bp} g^{\bp k}  & g^{\bl k}
\end{matrix}
\right).
\end{gather*}
Building on this, we give an adapted presentation of (\cite[Theorem 2.9]{Bismut}), which shows that the \emph{Chern connection} of the Hermitian metric $G$ is equivalent in a canonical way to the \emph{Bismut connection} of $g$, yielding explicit relationships between their curvature tensors.  Specifically, letting $\Omega^G$ denote the Chern curvature of $G$, we set
\begin{align*}
S^G_g = \i \tr_{\gw} \Omega^G \in \End(\QQ).
\end{align*}
As it turns out, the tensor $S^G_g$ is in a natural sense equivalent to the Bismut-Ricci form, and in particular
\begin{align*}
S^G_g \equiv 0 \qquad \Longleftrightarrow \qquad \rho_B \equiv 0.
\end{align*}
It follows that the existence of Bismut Hermitian-Einstein metrics requires Mumford-Takemoto slope stability of the associated holomorphic Courant algebroid.  The precise statement below is Corollary \ref{cor:stab}.  Going further we use this abstract slope stability obstruction to give a concrete obstruction in terms of holomorphic maps in Theorem \ref{thm:obstruction} and Corollary \ref{cor:obstructionfib}.  As a consequence we give the first examples in complex dimension $n \geq 3$ of compact complex manifolds which do not admit Bismut Hermitian-Einstein metrics, in fact we give infinite families in every dimension (cf. Examples \ref{ex:TbundleRiemann}, \ref{ex:genT2bundle}).  We state this for emphasis as our first main theorem:

\begin{thm} \label{t:nonexistence} In every dimension there exist infinitely many complex manifolds with vanishing first Chern class which do not admit a Bismut Hermitian-Einstein metric.
\end{thm}

Returning to the question of constructibility of solutions, and recalling the construction of Calabi-Yau metrics using K\"ahler-Ricci flow, it is natural to construct Bismut Hermitian-Einstein metrics by means of a geometric flow.  With the goal of constructing canonical geometries on complex manifolds in mind, the third author and Tian introduced a natural geometric flow of pluriclosed metrics, called pluriclosed flow \cite{PCF}.  This equation is expressed in terms of a pair $(\gw, \gb)$ as above via
\begin{align*}
\dt \gw =&\ - \rho_B^{1,1}, \qquad \dt \gb = - \rho_B^{2,0}.
\end{align*}
This equation is strictly parabolic, and admits short-time solutions with arbitrary initial data on compact manifolds.  Global existence and convergence results in various geometric settings have appeared in \cite{ASnondeg, ASUtoric, Arroyo, Boling, FinoFlow, StreetsGG, StreetsND, StreetsSTB}.  Using the discussion above, we see that it is possible to equivalently formulate pluriclosed flow in terms of a family of generalized Hermitian metrics $G$ on $\QQ$, which satisfy
\begin{align*}
G^{-1} \dt G =&\ - S^G_g.
\end{align*}
Viewed this way, pluriclosed flow bears a formal similarity to the Hermitian-Yang-Mills flow \cite{DonaldsonHYM}, although here there is further nonlinearity due to the fact that the time-dependent metric $g$ is used to take the trace defining $S^G_g$, whereas a background metric is used in HYM flow.  Precursors of this formulation of pluriclosed flow appeared in \cite{StreetsPCFBI, JordanStreets}, where the equation was shown to hold locally.  The key point in producing a general, global formula is the use of the twisted $\delb$-operator on $\QQ$.  Restricting to the setting of vanishing first Chern class, and in view of the nonexistence phenomena discussed above, two natural questions emerge.  First, \emph{if a Bismut Hermitian-Einstein metric exists, does the flow always converge to it?}  Second, \emph{if no Bismut Hermitian-Einstein metric exists, what geometric singularities does the flow encounter?}

To address the first question, let us restrict to what so far gives the only known examples of non-K\"ahler Bismut Hermitian-Einstein structures, namely when the associated Bismut connection is \emph{flat}.  First note that by the classical results of Cartan-Schouten (\cite{CartanSchouten2,CartanSchouten} cf. also \cite{AgricolaFriedrich}), it follows that the universal cover of such a manifold is isometric as a Riemannian manifold to a simply connected Lie group equipped with bi-invariant metric and flat Cartan connection.  The low-dimensional compact examples give some of the first known examples of complex, non-K\"ahler manifolds.  The classic Hopf/Boothby metric on standard Hopf surfaces is Bismut-flat (cf. Example \ref{e:standardHopf}), as are the standard metrics on Calabi-Eckmann spaces (cf. Example \ref{e:CCE}).  These are special cases of the construction of Samelson \cite{Samelson}, who showed generally that any compact Lie group of even dimension admits left-invariant complex structures compatible with a bi-invariant metric, and that these structures are Bismut-flat (cf. also \cite{AlexandrovIvanov}).  It was shown recently in \cite{Bismutflat} that all compact Bismut-flat Hermitian manifolds are covered by these `Samelson spaces'.

Our second main theorem gives a definitive picture of the global existence and convergence of pluriclosed flow on such backgrounds:
\begin{thm} \label{t:mainthm} Let $(M^{2n}, \gw_F, J)$ be a compact Bismut-flat manifold.  Given $\gw_0$ a pluriclosed metric so that $[\del \omega_0] = [\del \omega_F] \in H^{2,1}_{\delb}$, the solution to pluriclosed flow with initial data $\gw_0$ exists on $[0,\infty)$ and converges to a Bismut-flat metric $\gw_{\infty}$.
\end{thm}

In view of our description of pluriclosed flow in terms of metrics on holomorphic Courant algebroids, the hypothesis $[\del \gw_0] = [\del \omega_F] \in H^{2,1}_{\delb}$ is natural, as then the solution to pluriclosed flow exists on a fixed holomorphic Courant algebroid with the same underlying holomorphic vector bundle as our given background flat metric.  In many concrete cases, however, we obtain a complete description of the global existence and convergence by exploiting the structure of the underlying cohomologies. This is the case for instance in Examples \ref{e:standardHopf} and \ref{e:CCE} below, where we obtain a complete description of the global existence and convergence of the flow with arbitrary initial data on standard Hopf surfaces and Calabi-Eckmann spaces, in particular confirming part (1) of Conjecture 7.6 of \cite{SG}.  Furthermore, we illustrate several higher dimensional examples of compact Lie groups where Theorem \ref{t:mainthm} can be applied.

From a PDE perspective Theorem \ref{t:mainthm} illustrates interesting behavior.  Pluriclosed metrics are described locally by a $(1,0)$-form, and pluriclosed flow admits a reduction to a quasilinear degenerate parabolic \emph{system} of equations for a $(1,0)$-form (cf. \cite{StreetsPCFBI}).  Such systems tend to exhibit local singularities, such as the classic neckpinch singularities of Ricci flow and mean curvature flow, and yet Theorem \ref{t:mainthm}, describing the global behavior with arbitrary initial data on certain backgrounds, shows that such behaviors cannot occur for pluriclosed flow.  From a more geometric point of view, we note that while recently there have been some results on geometric flows of non-K\"ahler metrics converging to rigid, K\"ahler metrics (cf. \cite{LeeStreets, StreetsPCFBI, CRF1}), or converging to interesting non-K\"ahler metrics assuming a certain symmetric ansatz (\cite{PPZ, StreetsRYMPCF}), Theorem \ref{t:mainthm} seems to be the first result showing that a natural class of \emph{non-K\"ahler} metrics is globally attractive for a geometric flow with arbitrary initial data.

The proof of Theorem \ref{t:mainthm} relies principally on the formulation of pluriclosed flow in terms of the associated generalized Hermitian metric $G$ described above.  Using this we derive extensions of the parabolic Schwarz Lemma for K\"ahler-Ricci flow to the setting of pluriclosed flow, which now have reaction terms expressed in terms of the Chern curvature of a chosen background metric on $\QQ$, which in turn is equivalent to the Bismut curvature of a background pluriclosed metric $g$.  Assuming the background metric has vanishing Bismut curvature, this leads to uniform equivalence of the time-dependent metrics $G_t$, and furthermore to a Calabi/Yau type estimate for the Chern connection of $G$.  After this uniform $C^{\infty}$ estimates follow from Schauder theory, leading to global existence of the flow.  The a priori estimates imply subconvergence of the flow to a Bismut-flat metric along some sequence of times, and further maximum principle arguments set up against the background of this subsequential limit show convergence of the flow line.

Turning to the second question above, we can ask what geometric singularities pluriclosed flow encounters on manifolds with vanishing first Chern class but which do not admit a Bismut Hermitian-Einstein metric.  As discussed above, the first such examples occur for minimal non-K\"ahler complex surfaces of Kodaira dimension $\gk \geq 0$.  We use our methods above to prove the definitive long-time existence result in this setting:

\begin{thm} \label{t:kod1thm} Let $(M^4, J)$ be a minimal non-K\"ahler surface of Kodaira dimension $\gk \geq 0$.  Given $\gw_0$ a pluriclosed metric on $M$, the solution to pluriclosed flow with initial condition $\gw_0$ exists on $[0,\infty)$.
\end{thm}

The proof follows similar lines to Theorem \ref{t:mainthm} as described above, relying on a priori estimates for the metric derived from the Schwarz lemma.  It follows from Kodaira's results on elliptic fibrations that such manifolds are finitely covered by holomorphic principal $T^2$-fibrations over Riemann surfaces, and after passing to the relevant cover we use the fibration structure to obtain estimates over the horizontal components of the metric.  Using then special properties of the Bismut curvature tensor of a background $T^2$-invariant metric, we are able to derive maximum principles leading to uniform control over the whole metric.  This confirms the `cone conjecture' \cite[Conjecture 5.2]{PCFReg} for these manifolds.  It is expected that blowdown limits at infinity for this flow converge in Gromov-Hausdorff sense to the base of the fibration equipped with a constant curvature metric (resp. to a point) when $\gk = 1$ (resp. $\gk = 0$).  This has been confirmed for $T^2$-invariant initial data in \cite{StreetsRYMPCF} (cf. Remark \ref{r:convrmk}).

Theorem \ref{t:mainthm} also has applications to generalized K\"ahler (GK) geometry \cite{GualtieriGKG}.  A GK structure is a triple $(g, I, J)$ consisting of a Riemannian metric $g$ compatible with two integrable complex structures $I$ and $J$, with K\"ahler forms $\gw_I$ and $\gw_J$ satisfying
\begin{align*}
d^c_I \gw_I = H = - d^c_J \gw_J, \qquad dH = 0.
\end{align*}
A key example of generalized K\"ahler structure occurs for Bismut-flat structures.  In particular, if we fix $G$ a compact Lie group with bi-invariant metric $g$, and let $I$ and $J$ denote left- and right-invariant complex structures compatible with $g$, then the triple $(g, I, J)$ is generalized K\"ahler \cite{GualtieriGKG}.  Furthermore, there always exist infinite dimensional families of local deformations of a given generalized K\"ahler structure using natural classes of Hamiltonian flows \cite{Bischoff, GibsonStreets}.  These families are analogous to K\"ahler classes, and it is natural to seek classification of generalized K\"ahler structures up to equivalence via these deformations.  Given this, it is natural to ask if a left-invariant complex structure $I$ as above can be part of an `exotic' generalized K\"ahler structure, where $J$ does not arise via the construction above, up equivalence by Hamiltonian flow.  We note that already on the torus this question is nontrivial due to the construction of nonstandard orthogonal complex structures on flat tori in complex dimension $n \geq 3$ \cite{BSV}.  

From this perspective, the question of the topology of this space is loosely analogous to questions on the uniqueness of symplectic structures, in particular conjectures on whether the space of symplectic structures on a hyperK\"ahler four-manifold is connected, for which geometric flow methods have been proposed (cf. \cite{DonaldsonMM}).  More generally, the flow method has been employed in various settings to understand the global structure of natural classes of geometric structures (e.g. \cite{BK1, BHH,MarquesPSC}).  It turns out that the pluriclosed flow provides a very natural approach to understanding the nonlinear space of generalized K\"ahler structures.  As shown in \cite{GKRF}, the pluriclosed flow furthermore preserves the generalized K\"ahler condition, when coupled to an evolution equation for one of the complex structures.  This is called \emph{generalized K\"ahler-Ricci flow}, and is expressed in terms of the K\"ahler form $\gw_J$, a one-parameter family $\gb_J \in \gL^{2,0}_J$, and a one-parameter family of complex structures $I$ as
\begin{align*}
\dt \gw_J = - (\rho_B^J)^{1,1}, \qquad \dt \gb_J =&\ - (\rho_B^J)^{2,0}, \qquad \dt I = L_{\theta_I^{\sharp} - \theta_J^{\sharp}} I,
\end{align*}
where $\theta_J = J d^* \gw_J$ is the Lee form of the Hermitian structure $(g, J)$, with $\theta_I$ defined similarly.  In particular, the metrics $\gw_J$ are evolving by pluriclosed flow, and so Theorem \ref{t:mainthm} applies to give a global description of the space of generalized K\"ahler structures on Bismut-flat backgrounds.

\begin{cor} \label{c:GKcor} Let $(M^{2n}, \gw_F, J)$ be a compact Bismut-flat manifold.  If $(g, I, J)$ is a generalized K\"ahler structure on $M$ such that $[\del \gw_J] = [\del \gw_F] \in H^{2,1}_{\delb}$, then the solution to generalized K\"ahler-Ricci flow with this initial condition exists on $[0,\infty)$ and converges to a Bismut-flat structure $(g_{\infty}, I_{\infty}, J)$.  In particular, $I$ is biholomorphic to a complex structure compatible with a Bismut-flat metric.
\end{cor}

\noindent As a concrete example, we use Corollary \ref{c:GKcor} to classify all generalized K\"ahler structures on standard Hopf surfaces (cf. Example \ref{e:standardHopf}).

Here is an outline of the rest of the paper.  We begin in \S \ref{s:bck} by recalling background results from Hermitian geometry, and discuss general results on Bismut Hermitian-Einstein metrics.  In \S \ref{s:PCmetrics} we expound upon ideas of Bismut (cf. \cite{Bismut} \S IIb), describing pluriclosed metrics in a fixed Aeppli class on $M$ in terms of generalized Hermitian metrics $G$ on the canonically associated holomorphic Courant algebroid, with twisted $\delb$-operator.  Furthermore, we derive explicit formulas relating the Chern curvature of $G$ to the Bismut curvature tensor of $g$.  This provides an explicit connection to Hermitian-Yang-Mills theory, and we use this in \S \ref{s:slopestability} to provide slope-stability obstructions to the existence of Bismut Hermitian-Einstein metrics, and the existence of manifolds with $c_1 = 0$ which do not admit Bismut Hermitian-Einstein metrics.  We turn in \S \ref{s:PCF} to the implications for pluriclosed flow.  We first observe an extension of the Schwarz Lemma to the setting of holomorphic Courant algebroids.  Using this, we obtain natural parabolic differential inequalities for the generalized Hermitian metric along solutions to pluriclosed flow, leading then to the proofs of the global existence and convergence results.

\vskip 0.1in 

\noindent \textbf{Acknowledgements:} We thank Ra\'ul Gonz\'alez Molina and Dan Popovici for useful discussions.

\section{Background} \label{s:bck}
\subsection{Pluriclosed metrics}

Let $(M^{2n},J)$ be a complex manifold. A Riemannian metric on $(M,J)$ is Hermitian if $g(JX, JY) = g(X, Y)$ for $X, Y \in TM$. Given a Hermitian metric $g$ we define the associated K\"ahler form $\omega \in \Lambda^{1,1}_{\mathbb R}$ by
\begin{align*}
\omega(X, Y) = g(J X, Y).
\end{align*}
A Hermitian metric $g$ on $(M,J)$ is called \emph{pluriclosed} if
$$
2\i \del\delb \gw = dd^c \omega = 0,
$$
where $d^c = \i (\delb - \del)$ is the conjugate differential, and in particular $d^c \gw = - d \gw(J, J, J)$.  Notice that a pluriclosed Hermitian metric canonically defines a pair of cohomological classes
$$
[\omega] \in H^{1,1}_A, \qquad [\partial\omega] \in H^{2,1}_{\overline{\partial}}
$$
which play a significant role in the present work. Here $H^{p,q}_{\overline{\partial}}$ and $H^{p,q}_A$ denote the Dolbeault and Aeppli cohomology groups of $(M,J)$, respectively, the latter being defined by
\begin{align*}
H^{p,q}_{A} := \frac{\left\{ \Ker \del\delb : 
\Lambda^{p,q} \to \Lambda^{p+1,q+1}\right\}}{\left\{\del 
\alpha + \delb \gamma\ |\ \ga + \gamma \in \Lambda^{p-1,q} \oplus \Lambda^{p,q-1} \right\}}.
\end{align*}
A Hermitian manifold $(M,g,J)$ has a canonical one-parameter family of Hermitian connections \cite{Gauduchonconn}, i.e. connections for which $g$ and $J$ are parallel.  This line of canonical connections is uniquely determined by the following two classical connections.

\begin{defn} \label{d:Chernconn} Let $(M^{2n}, g, J)$ be a Hermitian manifold.  Let $\N$ denote the Levi-Civita connection associated to $g$.  The \emph{Chern connection} is defined by
\begin{align*}
\IP{\N^C_X Y, Z} = \IP{\N_X Y, Z} - \tfrac{1}{2} d \gw (JX, Y, Z).
\end{align*}
The \emph{Bismut connection} is defined by
\begin{align*}
\IP{\N^B_X Y, Z} = \IP{\N_X Y, Z} - \tfrac{1}{2} d^c \gw(X,Y,Z).
\end{align*}
\end{defn}

Given a Hermitian connection $\nabla'$ on $(M^{2n}, g, J)$, we let $R^{\nabla'}$ denote the associated $(3,1)$-curvature tensor. That is,
\begin{align*}
R^{\N'}(X,Y)Z =&\ \N'_X \left(\N'_Y Z \right) - \N'_Y \left(\N'_X Z \right) - \N'_{[X,Y]} Z.
\end{align*}
We furthermore adopt the standard notation for the $(4,0)$-curvature tensor,
\begin{align*}
R^{\N'}(X,Y,Z,W) =&\ \IP{R^{\N'}(X,Y)Z,W}.
\end{align*}

\begin{defn} \label{d:Riccidef} Let $(M^{2n}, g, J)$ be a Hermitian manifold endowed with the Bismut connection $\N^B$, and set $R^B = R^{\N^B}$.  Define the \emph{first Ricci curvature} $\rho_B$ via
\begin{align*}
\rho_B(X,Y) := \tfrac{1}{2} \IP{ R^{B}(X,Y) J e_i, e_i},
\end{align*}
where $\{e_i\}$ is any orthonormal basis for the tangent space at a given point.
\end{defn}

To finish this section, we recall a classical curvature identity which is crucial to our discussion.  In particular, there is a relationship between the curvature tensor of the Bismut connection $\nabla^B$, and that of the metric compatible connection $\N^-$, defined by
\begin{align}\label{eq:Nablaminus}
\IP{\N^-_X Y, Z} = \IP{\N_X Y, Z} + \tfrac{1}{2} d^c \gw(X,Y,Z).
\end{align}

\begin{prop} \label{p:curvatureflip} (cf. \cite[Theorem 1.6]{Bismut})  \label{p:Bismutpair} Let $(M^{2n}, g, J)$ be a pluriclosed manifold.  Then, setting $R^- = R^{\N^-}$, one has
\begin{align*}
R^B(X,Y,Z,W) = R^-(Z,W,X,Y).
\end{align*}
\end{prop}

\subsection{Bismut Hermitian-Einstein metrics}

In this subsection we record some fundamental points regarding Bismut Hermitian-Einstein metrics.  We begin by formally stating the definition.

\begin{defn} Given a complex manifold $(M^{2n}, J)$, a pluriclosed metric $g$ on $M$ is \emph{Bismut Hermitian-Einstein} if
\begin{align*}
\rho_B \equiv 0.
\end{align*}
\end{defn}

To better understand this condition, we first give a reformulation in terms of the Riemannian Ricci tensor, using an identity for the Bismut curvature.

\begin{prop} \label{p:BismutRicci} \cite[Proposition 8.10]{garciafernStreets2020} Let $(M^{2n}, g, J)$ be a pluriclosed manifold, with $H = - d^c \gw$, and $\theta = - d^* \omega \circ J$ is the Lee form.  Then
\begin{gather*}
 \begin{split}
  \rho^{1,1}_B(\cdot, J \cdot) =&\ \Rc - \tfrac{1}{4} H^2 + \tfrac{1}{2} L_{\theta^{\sharp}} g,\\
  \rho_B^{2,0 + 0,2}(\cdot, J \cdot) =&\ - \tfrac{1}{2} d^* H + \tfrac{1}{2} d \theta - \tfrac{1}{2} i_{\theta^{\sharp}} H.
 \end{split}
\end{gather*}
\end{prop}

From this proposition it follows that Bismut Hermitian-Einstein structures satisfy a natural coupling of the classic Einstein equation and the equation for a harmonic three-form.  Moreover, it follows that they are automatically \emph{steady generalized Ricci solitons} (cf. \cite[Ch. 4, Proposition 8.14]{garciafernStreets2020}).  We next record a further consequence arising from the variational theory of the generalized scalar curvature functional (Perelman $\FF$-functional for generalized Ricci flow).

\begin{prop} \label{p:solitons} Suppose $(M^{2n}, J)$ is a compact complex manifold, and $g$ is a pluriclosed Bismut Hermitian-Einstein metric.  Let $u = e^{-f/2} \in C^{\infty}(M)$ be the first eigenfunction of the Schr\"odinger operator $-4 \gD + R - \tfrac{1}{12} \brs{H}^2$.  Then
\begin{gather} \label{f:solitons}
\begin{split}
\Rc - \tfrac{1}{4} H^2 + \N^2 f =&\ 0\\
d^*H + i_{\N f} H =&\ 0.
\end{split}
\end{gather}
Consequently, the vector fields $V = \tfrac{1}{2} \left( \theta^{\sharp} - \N f \right)$ and $IV$ are both Killing and holomorphic and $V$ satisfies
\begin{gather}\label{f:genisom}
\tfrac{1}{2} d \theta - i_V H = 0.
\end{gather}
Furthermore, if $V$ vanishes identically then $g$ is K\"ahler Ricci-flat.
\begin{proof} 
Given $g$ as in the statement, define $H = - d^c \gw$ and $\theta = - d^* \omega \circ J$ as above.  Let $(g_t, H_t)$ denote the solution to generalized Ricci flow with initial data $(g, H)$.  By the identities of Proposition \ref{p:BismutRicci}, one has $(g_t, H_t) = (\phi_t^* g, \phi_t^* H)$, where $\phi_t$ is the one-parameter family of diffeomorphisms generated by $\theta^{\sharp}$.  Due to the monotonicity formula for the $\FF$ functional along generalized Ricci flow (cf. \cite[Proposition 6.8, Corollary 6.11]{garciafernStreets2020}), the existence of $f$ satisfying \eqref{f:solitons} as claimed follows. Combining now equation \eqref{f:solitons} with Proposition \ref{p:BismutRicci}, $V$ must be Killing and furthermore it satisfies \eqref{f:genisom}. Arguing now as in the proof of \cite[Proposition 4.1]{SU2}, the second equation in \eqref{f:solitons} implies
$$
(L_V\omega)^{2,0} = 0 = (L_{IV} \omega)^{1,1}.
$$
Using that $V$ is Killing, it follows that $L_V I = 0$, and hence $L_{IV} \omega = (L_{IV} \omega)^{1,1} = 0$. Thus, since $IV$ is also holomorphic, we conclude that $L_{IV}g = 0$. 

Finally, if $V$ vanishes identically then one has $\theta = -df$. Since $g$ is pluriclosed, if we assume $n >2$ it follows from \cite[Theorem 1.3]{IvanovPapadopoulos} that $g$ is K\"ahler. In the case $n = 2$, by the conformal transformation law for the Lee form the conformally related Hermitian metric $e^fg$ is K\"ahler. But by the uniqueness of the Gauduchon metric in a fixed conformal class, it follows that, after possibly modifying $f$ by a constant, $g = e^f g$ and so $f \equiv 0$. Thus for any choice of $n$ the metric $g$ is K\"ahler, and Ricci flat.

\end{proof}
\end{prop}

As noted in the introduction, the system of equations (\ref{f:solitons}) are the same as the equations of motion in compactifications of type II supergravity \cite{Friedanetal}.  Classification results for solutions to this more general soliton system on complex surfaces have appeared recently in \cite{Streetssolitons, SU1}, with a new class of complete solutions arising from an extension of the Gibbons-Hawking ansatz recently appearing \cite{SU2}.  The starting point is the observation that the vector field $V = \tfrac{1}{2} \left( \theta^{\sharp} - \N f \right)$ is real holomorphic, and $JV$ is Killing (cf. \cite[Proposition 4.1]{SU2}). It is interesting to notice that the Killing condition $L_V g = 0$ combined with \eqref{f:genisom} yields a \emph{generalized isometry} on an exact Courant algebroid determined by  the torsion $H$ (cf. \cite[Proposition 2.53]{garciafernStreets2020}).

\subsection{Bismut-flat structures}

An important class of Bismut Hermitian-Einstein structures occurs when the associated Bismut connection is actually flat. As a matter of fact, Bismut-flat manifolds are the only known examples of non-K\"ahler Bismut Hermitian-Einstein metrics to the present day. In this section we recall fundamental results on the structure of compact pluriclosed Bismut-flat manifolds. We will also provide concrete examples where Theorem \ref{t:mainthm} and Corollary \ref{c:GKcor} apply. To begin we recall two basic examples, namely the standard Hopf surfaces, and Calabi-Eckmann threefolds.

\begin{ex} \label{e:standardHopf} Given complex numbers $\ga, \gb$ satisfying $\brs{\ga} \leq \brs{\gb} < 1$, we obtain a \emph{Hopf surface} via
\begin{align*}
M_{\ga\gb} = \mathbb C^2 \backslash  \{\mathbf{0}\} / \IP{ (z_1,z_2) \to (\ga z_1, \gb z_2)}.
\end{align*}
For any $\ga, \gb$ this is a complex manifold diffeomorphic to $S^3 \times S^1$.  We furthermore say that it is \emph{standard} if $\brs{\ga} = \brs{\gb}$.  In this case there is a natural pluriclosed metric known as the Hopf/Boothby metric defined by the invariant K\"ahler form on $\mathbb C^2 \backslash \{\mathbf{0}\}$ by
\begin{align*}
\gw_{\mbox{\tiny{Hopf}}} = \frac{\i}{\brs{z}^2} \left( dz_1 \wedge d \bar{z}_1 + d z_2 \wedge d \bar{z}_2 \right).
\end{align*}
This metric is the unique (up to scaling) bi-invariant Hermitian metric for $S^3 \times S^1 \cong SU(2) \times U(1)$, and the associated Bismut connection is flat.  Recall also that $h^{1,1}_{A}(M_{\ga\gb}) = 1$ (see the proof of \cite[Theorem 3.3]{AngellaTomassini}) and that the Aeppli class of a pluriclosed Hermitian metric on a compact complex surface is non-zero (see e.g. \cite{PiovaniTomassini}). Therefore, given any pluriclosed metric $\omega$ on $M_{\ga\gb}$ one has $[\omega] = \lambda [\gw_{\mbox{\tiny{Hopf}}}]$ for $0 < \lambda \in \mathbb{R}$, where the positivity of $\lambda$ follows e.g. by integration on a holomorphic curve (cf. below). Using now that $[\partial \omega] \in H^{2,1}_{\overline{\partial}}$ factorizes through the natural map
$$
H^{1,1}_A \to H^{2,1}_{\overline{\partial}} \colon [\omega] \mapsto [\partial \omega],
$$
we have that Theorem \ref{t:mainthm} applies giving convergence of the pluriclosed flow to a Bismut-flat structure for any initial data.  Due to the classification by Gauduchon-Ivanov \cite{GauduchonIvanov} (cf. also \cite[Theorem 8.26]{garciafernStreets2020}), it follows that $\gw_{\mbox{\tiny{Hopf}}}$ is the unique Bismut-flat metric in its Aeppli class, thus the limiting metric is a scalar multiple of the Hopf metric.  

The Hopf metric above is part of two distinct types generalized K\"ahler structures. In the first case, the second compatible complex structure $I$ is obtained by modifying the complex structure on $\mathbb{C}^2$ such that $(z_1, \overline{z_2})$ are holomorphic coordinates. Note that, so defined, $I$ is also Hermitian with respect to $g_{\mbox{\tiny{Hopf}}} = \gw_{\mbox{\tiny{Hopf}}}(,J)$. Furthermore, $I$ and $J$ have opposite orientations and they satisfy $[I,J] = 0$. The possible generalized K\"ahler structures of the form $(g_{\mbox{\tiny{Hopf}}}, I, J)$ where $I$ and $J$ induce opposite orientations are classified in \cite[Theorem 3]{ApostolovGualtieri}. In the second case \cite{GualtieriGKG}, consider the compatible complex structure $I'$ defined by the  holomorphic $(2,0)$-form
$$
\Omega_{I'} = (\overline{z}_1dz_1 + z_2 d\overline{z}_2) \wedge (\overline{z}_1dz_2 - z_2 d\overline{z}_1).
$$
Then $(g_{\mbox{\tiny{Hopf}}},I',J)$ is a generalized K\"ahler structure whose associated Poisson tensor $\gs = \tfrac{1}{2}[I,J]$ is generically nondegenerate but degenerates along the pair of elliptic curves $\{z_1 = 0\}, \{z_2 = 0\}$.  The possible generalized K\"ahler structures of the form $(g_{\mbox{\tiny{Hopf}}}, I, J)$ where $I$ and $J$ induce the same orientation form a finite dimensional family and are classified by twistor constructions in \cite{GauduchonWeyl} (cf. also \cite{PontecorvoCS}).

Now suppose we have an arbitrary generalized K\"ahler structure $(M, g, I, J)$ where $(M, J) \cong M_{\ga\gb}$.
Applying now Corollary \ref{c:GKcor} and the discussion above, the generalized K\"ahler Ricci flow with initial data $(g,I,J)$ exists for all times and converges to a generalized K\"ahler structure $(g_{\mbox{\tiny{Hopf}}},I_{\infty},J)$ with $I_\infty$ biholomorphic to $I$.  Thus the generalized K\"ahler-Ricci flow deforms all generalized K\"ahler structures on $M_{\ga\gb}$ to the space of GK structures of the form $(g_{\mbox{\tiny{Hopf}}}, I, J)$. Therefore, a classification of generalized K\"ahler structures of the form $(g, I, J)$ where $(M, J) \cong M_{\ga\gb}$ follows from \cite{ApostolovGualtieri,GauduchonWeyl}.
\end{ex}

\begin{ex} \label{e:CCE}
Consider $\mathbb C^2 \backslash \{0\} \times \mathbb C^2 \backslash \{0\}$ endowed with the $\mathbb C$-action given by
\begin{align*}
\gg(z, w) =  (e^{\gg} z, e^{\sqrt{-1} \gg} w).
\end{align*}
This action is free and proper, the quotient space $(M,J)$ is diffeomorphic to $S^{3} \times S^{3}$, and this is an example of a Calabi-Eckmann space.  Furthermore $M$ is the total space of a holomorphic $T^2$ fibration over $\mathbb {CP}^1 \times \mathbb {CP}^1$.  This is the product of the standard Hopf fibrations $\pi_i : S^3 \to \mathbb{CP}^1$.  Let $\xi_i$ denote the canonical vector fields associated to this fibration on the two factors, with $\mu_i$ the associated canonical connections satisfying $d \mu_i = \pi_i^* \gw_{FS}$.  Note that the complex structure further satisfies $J \xi_1 = \xi_2$.  We consider the K\"ahler form associated to the product of round metrics, namely
\begin{align*}
\gw = \pi_1^* \gw_{FS} + \pi_2^* \gw_{FS} + \mu_1 \wedge \mu_2.
\end{align*}
It follows that
\begin{align*}
d \gw =&\ d \mu_1 \wedge \mu_2 - \mu_1 \wedge d \mu_2 = \pi_1^* \gw_{FS} \wedge \mu_2 - \mu_1 \wedge \pi_2^* \gw_{FS},
\end{align*}
hence
\begin{align*}
H = - d^c \gw =&\ d \gw(J, J, J) = \pi_1^* \gw_{FS} \wedge \mu_1 - \pi_2^* \gw_{FS} \wedge \mu_2.
\end{align*}
It is clear then that the geometric structure is that of a product of two copies of the $S^3$ factor of the Hopf metric, and furthermore that $d H = 0$. It can be checked that $(M,J,\omega)$ is indeed Bismut-flat, and that $\theta^{\sharp}$, where $\theta = -\tfrac{1}{2} (\mu_1 + \mu_2)$, is a Killing field which further preserves $H$. We refer to \cite[Example 8.35]{garciafernStreets2020} for further details on the geometry.

For the Calabi-Eckmann threefold one has $h^{1,1}_{A} = 1$ (see the proof of \cite[Proposition 3.4]{AngellaTomassini}). Furthermore, it follows by integrating over the $T^2$ fibers that for an arbitrary pluriclosed metric on $M$ one has
$$
0 \neq [\omega'] \in H^{1,1}_A,
$$
and hence there exists $\gl > 0$ so that $[\gw'] = \gl [\gw]$.  Arguing now as in Example \ref{e:standardHopf}, we obtain that
$$
[\partial \omega'] = \lambda [\partial \omega] \in H^{2,1}_{\overline{\partial}},
$$
and hence Theorem \ref{t:mainthm} applies giving global existence and convergence of the pluriclosed flow to a Bismut-flat structure for any initial data.
\end{ex}

These examples are both covered by the general construction of Samelson \cite{Samelson}, which we recall next. Let $G$ denote a compact connected real Lie group of dimension $2n$. We fix a bi-invariant metric $\IP{,}$ on the Lie algebra $\mathfrak{g} := \operatorname{Lie} G$ inducing a bi-invariant metric $g$ on $G$. Any choice of integrable left-invariant complex structure $J_L$ on $G$ compatible with $g$ determines a Bismut-flat structure on $G$, with Bismut torsion given by the Cartan three-form
$$
H(X,Y,Z) = -d^c_{J_L}\omega(X,Y,Z) = g([X,Y],Z),
$$
for any triple of left-invariant vector fields $X,Y,Z$ on $G$. Furthermore, denote by $J_R$ the right-invariant extension of the almost complex structure induced by $J_L$ on $\mathfrak{g}$. Then, the triple $(g,J_L,J_R)$ is a natural generalized K\"ahler structure on $G$ \cite{GualtieriGKG}. In particular, Example \ref{e:standardHopf} can be recovered from this construction for $G = SU(2) \times U(1)$, while Example \ref{e:CCE} can be recovered taking $G = SU(2) \times SU(2)$ (for suitable choices of $J_L$).

To the knowledge of the authors, in general there is no available calculation of the cohomology groups $H^{2,1}_{\overline{\partial}}$ and $H^{1,1}_A$ in the construction of Samelson. Hence our argument in the previous two examples, which grants the global existence of the pluriclosed flow for any initial data, does not apply directly. Note that, for $G$ compact and simply connected, a calculation of the Dolbeault cohomology for an arbitrary left-invariant complex structure was undertaken in \cite[Proposition 5.9]{Pittie}. 
In order to discuss a pair of concrete examples, we need a more explicit description of Samelson's construction. An integrable left-invariant complex structure on $G$ compatible with $g$ is equivalent to a choice of complex Lie subalgebra 
$$
\mathfrak{s} \subset \mathfrak g^c := \mathfrak g \otimes \mathbb{C}
$$
satisfying
$$
\mathfrak{s} \cap \mathfrak{g} = \{0\}, \qquad \dim_{\mathbb{C}} \mathfrak{s} = n, \qquad \IP{\mathfrak{s},\mathfrak{s}} = 0,
$$
where the last condition means that $\mathfrak{s}$ is isotropic with respect to the $\mathbb{C}$-linear extension of $\IP{,}$. This data is actually equivalent to a choice of maximal torus $T \subset G$ with Lie algebra $\mathfrak{t}$, and a choice of isotropic subspace
$$
\mathfrak{a} \subset \mathfrak{t}^c := \mathfrak t \otimes \mathbb{C}
$$
such that $\mathfrak{a} \cap \mathfrak{t} = \{0\}$ (see e.g. \cite{Pittie}).  

\begin{ex} \
Consider $G = SU(3)$ endowed with the bi-invariant metric $g$ given by minus the Killing form. Let $J_L$ be a $g$-compatible integrable left-invariant complex structure on $G$ with associated maximal torus $T^2 \subset G$ and (isotropic) one-dimensional abelian subalgebra
$$
\mathfrak{a} \subset \mathfrak{t}^c.
$$
Then, there is an isomorphism of bigraded algebras \cite[Section 4]{Pittie}
$$
H^{*,*}_{\overline{\partial}}(G,J_L) \cong \Lambda^* \IP{u_{0,1},u_{2,1}} \otimes \mathbb{C}[u_{1,1}]/\langle u_{1,1}^3 \rangle
$$
where $u_{p,q}$ has bi-degree $(p,q)$ . In particular
$$
H^{2,1}_{\overline{\partial}}(G,J_L) \cong \mathbb{C}\IP{u_{2,1}} \cong \mathbb{C}.
$$
Therefore, Theorem \ref{t:mainthm} applies  giving convergence of the pluriclosed flow to a Bismut-flat structure for any initial data $\omega_0$ such that
\begin{equation}\label{eq:ray}
[\partial \omega_0] \in \mathbb{R}_{>0} \cdot [\partial \omega] \subset \mathbb{C}.
\end{equation}
Similarly, applying Corollary \ref{c:GKcor}, the generalized K\"ahler Ricci flow with such initial data $(g_0,I_0, J_L)$ exists for all times and converges to a generalized K\"ahler structure $(g_\infty,I_{\infty}, J_L)$ with Bismut flat metric $g_\infty$ and $I_\infty$ biholomorphic to $I_0$.
\end{ex}

\begin{ex}\label{ex:SO(9)}
Consider $G = SO(9)$ endowed with the bi-invariant metric $g$ given by minus the Killing form. Choose a maximal torus $T^4 \subset G$ and an orthonormal basis for its Lie algebra
$$
\mathfrak{t} = \langle e_1,e_2,e_3,e_4 \rangle.
$$
Consider the union of twelve lines in $\mathbb{P}(\mathfrak{t}^c)$ defined by the equations
$$
\sum_j x_j^2 = 0, \qquad \sum_j x_j^6 = 0,
$$
where $x = \sum_j x_j e_j \in \mathfrak{t}^c$. A choice of such line determines an isotropic subspace $\mathfrak{a} \subset \mathfrak{t}^c$ with no real points, and hence it defines a $g$-compatible integrable left-invariant complex structure $J_L$ on $G$. Then, there is an isomorphism of bigraded algebras \cite[p. 149]{Pittie}
$$
H^{*,*}_{\overline{\partial}}(G,J_L) \cong \Lambda^* \mathfrak{a}^* \otimes \Lambda^*(u_{2,1},u_{5,6}) \otimes \mathbb{C}[u_{1,1}, v_{1,1}]/\langle u_{1,1}^4 + v_{1,1}^4, u_{1,1}^4 v_{1,1}^4 \rangle
$$
where $u_{p,q}$ has bi-degree $(p,q)$ and the elements in $\mathfrak{a}^*$ have bi-degree $(0,1)$. In particular
$$
H^{2,1}_{\overline{\partial}}(G,J_L) \cong \mathbb{C}\IP{u_{2,1}} \cong \mathbb{C}.
$$
Therefore, Theorem \ref{t:mainthm} and Corollary \ref{c:GKcor} apply also in this example, for any initial data $\omega_0$ such that \eqref{eq:ray} is satisfied.
\end{ex}


%
%
%
%
\section{Pluriclosed metrics and holomorphic Courant algebroids} \label{s:PCmetrics}

In this section we show that a pluriclosed Hermitian structure yields a natural notion of Hermitian metric $G$ on a holomorphic Courant algebroid. We then use different methods to rederive a striking relationship--originally observed in \cite{Bismut}--between the curvature of the Chern connection associated to $G$ and the curvature of the Bismut connection associated to the underlying pluriclosed structure. Our results extend the study of metrics on holomorphic Courant algebroids initiated in \cite{garciafern2020gauge,garciafern2018holomorphic,StreetsPCFBI}.

\subsection{Holomorphic Courant algebroids}
Let $(M,J)$ be a complex manifold. We denote by $\mathcal{O}_M$ and $\underline{\mathbb{C}}$ the sheaves of holomorphic functions and $\mathbb{C}$-valued constant functions on $M$, respectively.

\begin{defn}\label{d:CAhol} 
A \emph{holomorphic Courant algebroid} is a holomorphic vector bundle $\mathcal{Q} \to (M,J)$ together with a nondegenerate holomorphic symmetric bilinear form $\IP{,}$, a holomorphic vector bundle morphism $\pi:\mathcal{Q}\to T^{1,0}$ called anchor map, and a Dorfman bracket on holomorphic sections of $\mathcal{Q}$, that is, a homomorphism of sheaves of $\underline{\mathbb{C}}$-modules
$$
[ \cdot,\cdot ] \colon \mathcal{Q} \otimes_{\underline{\mathbb{C}}} \mathcal{Q} \to \mathcal{Q},
$$
satisfying, for $u,v,w\in \QQ$ and $\phi\in \mathcal{O}_M$,
  \begin{itemize}
  \item[(D1):] $[u,[v,w]] = [[u,v],w] + [v,[u,w]]$,
  \item[(D2):] $\pi([u,v])=[\pi(u),\pi(v)]$,
  \item[(D3):] $[u,\phi v] = \pi(u)(\phi) v + \phi[u,v]$,
  \item[(D4):] $\pi(u)\IP{v,w} = \IP{[u,v],w} + \IP{v,[u,w]}$,
  \item[(D5):] $[u,v]+[v,u]=2\pi^* d\IP{u,v}$.
  \end{itemize}
\end{defn}

We will denote a holomorphic Courant algebroid $(\mathcal{Q},\IP{,},[\cdot,\cdot],\pi)$ simply by $\mathcal{Q}$. Using the isomorphism $\IP{,} \colon \mathcal{Q} \to \mathcal{Q}^*$ we obtain a holomorphic sequence
\begin{equation}\label{eq:holCouseqaux}
T^*_{1,0} \overset{\pi^*}{\longrightarrow} \mathcal{Q} \overset{\pi}{\longrightarrow} T^{1,0}.
\end{equation}

\begin{defn}\label{d:CAholexact} 
We will say that $\mathcal{Q}$ is \emph{exact} if \eqref{eq:holCouseqaux} induces an exact sequence of holomorphic vector bundles
\begin{equation}\label{eq:holCouseq}
0 \longrightarrow T^*_{1,0} \overset{\pi^*}{\longrightarrow} \mathcal{Q} \overset{\pi}{\longrightarrow} T^{1,0} \longrightarrow 0.
\end{equation}
\end{defn}

A complete classification of exact holomorphic Courant algebroids was obtained by Gualtieri in \cite{GualtieriGKG}. Let us summarize the result which we will use.

\begin{thm} \label{t:holCourant} 
Let $(M,J)$ be a complex manifold. Denote by $\Lambda^{2,0}_{cl}$ the sheaf of closed $(2,0)$-forms on $M$. Then, the set of isomorphism classes of exact holomorphic Courant algebroids on $M$ is bijective to the first \v Cech cohomology $H^1(\Lambda^{2,0}_{cl})$. Furthermore, there is a vector space isomorphism
\begin{equation}\label{eq:kcohomologyleqk}
H^1(\Lambda^{2,0}_{cl}) \cong \frac{\Ker \; d \colon \Lambda^{3,0 + 2,1}  \to \Lambda^{4,0 + 3,1 + 2,2}}{ \operatorname{Im} \; d \colon \Lambda^{2,0} \to \Lambda^{3,0 + 2,1}}.
\end{equation}
\end{thm}

A one-cocycle for the sheaf $\Lambda^{2,0}_{cl}$ determines local data for gluing the model $T^{1,0} \oplus T^*_{1,0}$ by means of holomorphic $B$-field transformations. On the other hand, a representative of a cohomology class in the right hand side of \eqref{eq:kcohomologyleqk} yields a convenient global description of a holomorphic Courant algebroid.

\begin{defn}\label{def:Q0}
Let $(M,J)$ be a complex manifold. Given $\tau \in \Lambda^{3,0 + 2,1}$, $d \tau = 0$, we denote by 
	\begin{equation*}
	\mathcal{Q}_{\tau} = T^{1,0} \oplus T^*_{1,0}
	\end{equation*}
	the exact holomorphic Courant algebroid with Dolbeault operator
	$$
	\overline{\partial}^\tau (X + \xi)  = \overline{\partial} X + \overline{\partial} \xi - i_{X}\tau^{2,1},
	$$ 
	symmetric bilinear form
\begin{equation}\label{eq:bilinear}
	\IP{X + \xi, X + \xi}  = \xi(X),
\end{equation}
	bracket on holomorphic sections given by
	\begin{equation*}
	[X + \xi, Y + \eta]_\tau   = [X,Y] + \partial (\eta(X)) +  i_X \partial \eta - i_Y \partial \xi + i_Yi_X \tau^{3,0},
	\end{equation*}
	and anchor map $\pi(X + \xi) = X$. It is not difficult to check that $\mathcal{Q}_{\tau}$ so defined satisfies the axioms in Definition \ref{d:CAhol}.
\end{defn}

The holomorphic Courant algebroids of our main interest arise via a reduction mechanism from an exact Courant algebroid in the smooth category. Recall that a smooth exact Courant algebroid $E$ over a smooth manifold $M$ is given by an exact sequence 
\begin{align*}
0 \longrightarrow T^*M \overset{\pi^*}{\longrightarrow} E \overset{\pi}{\longrightarrow} TM \longrightarrow 0.
\end{align*}
satisfying a set of axioms analogous to those in Definition \ref{d:CAholexact}. A choice of isotropic splitting of the previous sequence determines uniquely a closed real three-form $H \in \Lambda^3$ and an isomorphism 
\begin{equation*}
	E \cong TM \oplus T^*M,
	\end{equation*}
such that the symmetric bilinear form $\IP{,}$ is as in \eqref{eq:bilinear}, the anchor map is $\pi(X + \xi) = X$, and the bracket is
	given by
	\begin{equation*}
	[X + \xi, Y + \eta]_H   = [X,Y] + L_X \eta - i_Y d \xi + i_Yi_X H.
	\end{equation*}
Consider the smooth complex Courant algebroid $ E \otimes \mathbb{C}$, with Courant structure given by the $\mathbb{C}$-linear extensions of the symmetric bilinear form  $\IP{,}$, the bracket $[\cdot, \cdot]_H$, and the anchor map $\pi$. 

\begin{defn}\label{def:lifting}
Let $E$ be a smooth exact Courant algebroid over a complex manifold $(M,J)$. A \emph{lifting} of $T^{0,1}$ to $E \otimes \mathbb{C}$ is an isotropic, involutive subbundle $\ell \subset E \otimes \mathbb{C}$ mapping isomorphically to $T^{0,1}$ under the $\mathbb{C}$-linear extension of the anchor map $\pi \colon E \otimes \mathbb{C} \to T \otimes \mathbb{C}$.
\end{defn}

A lifting relates to the complex Courant algebroid $E \otimes \mathbb{C}$ as a Dolbeault operator relates to a smooth complex vector bundle, in the sense that it enable us to construct a Courant algebroid in the holomorphic category out of $E \otimes \mathbb{C}$. More precisely, following \cite{GualtieriGKG} we consider the reduction of $E \otimes \mathbb{C}$ by $\ell$ given by the orthogonal bundle 
$$
\mathcal{Q}_\ell := \ell^\perp/\ell,
$$
where $\ell^\perp$ is the orthogonal complement of $\ell$ with respect to the symmetric pairing on $E \otimes \mathbb{C}$. Since $\ell$ is a lifting of $T^{0,1}$ the kernel of $\pi_{|\ell^\perp}$ is $T^*_{1,0}$, and therefore $\mathcal{Q}_\ell$ is an extension of the form \eqref{eq:holCouseq}. The Dolbeault operator on $\mathcal{Q}_\ell$ is defined as follows: given $s$ a smooth section of $\mathcal{Q}_\ell$, we define
$$
\overline{\partial}^{\ell}_X s = [\tilde X,\tilde s] \quad \textrm{mod}\ \ell
$$
where $X \in T^{0,1}$, $\tilde X$ is the unique lift of $X$ to $\ell$, and $\tilde s$ is any lift of $s$ to a section of $\ell^\perp$. The Jacobi identity for the Dorfman bracket on $E \otimes \mathbb{C}$ implies that $\overline{\partial}^{\ell} \circ \overline{\partial}^{\ell} = 0$ and that it
induces a Dorfman bracket on the holomorphic sections of $\mathcal{Q}_\ell$.

Our next goal is to make the previous construction more explicit by choosing an isotropic splitting of $E$. For a proof of the next result we refer to \cite{garciafernStreets2020}.

\begin{lemma}\label{l:lifting} 
Let $(M,J)$ be a complex manifold. Given $H_0 \in \Lambda^3$ a closed real three-form, $d H_0 = 0$, consider the exact Courant algebroid $(T\oplus T^*, \IP{,}, [,]_{H_0},\pi)$ as above. Then, a lifting 
$$
\ell \subset (T \oplus T^*) \otimes \mathbb{C}
$$ 
of $T^{0,1}$ is equivalent to a pair $(\omega,b)$, where $\omega \in \Lambda^{1,1}_{\mathbb{R}}$ and $b \in \Lambda^2$, satisfying
\begin{equation}\label{eq:liftingcond}
H_0 = - d^c \omega - db.
\end{equation}
More explicitly, given $(\omega,b)$ satisfying \eqref{eq:liftingcond} the lifting is
	\begin{equation}\label{eq:L}
	\ell = \ell(\omega,b) := \{e^{b + \sqrt{-1}\omega}(X^{0,1}), \; X^{0,1} \in T^{0,1} \},
	\end{equation}
and, conversely, any lifting is uniquely expressed in this way.
\end{lemma}

\begin{rmk}
Notice that the inclusion $E \subset E \otimes \mathbb{C}$ defines a \emph{compact form} in the sense of \cite[Definition 5.4]{garciafern2020gauge}. Hence, the analogy between liftings of $T^{0,1}$ and Dolbeault operators suggests that the previous result can be regarded as a \emph{Chern correspondence} in our setting. We refer to  \cite{garciafern2020gauge} for further insights on this relation in a more general setup.
\end{rmk}

Our next result gives an explicit formula for the exact holomorphic Courant algebroid associated to a lifting $\ell(\omega,b)$.

\begin{lemma}\label{l:liftingQ} 
Let $(M,J)$ be a complex manifold endowed with an exact Courant algebroid $(T\oplus T^*, \IP{,}, [,]_{H_0},\pi)$ as above. Let $\ell(\omega,b)$ be a lifting of $T^{0,1}$ as in Lemma \ref{l:lifting}. Then, using the notation in Definition \ref{def:Q0}, there is a canonical isomorphism
$$
\mathcal{Q}_{\ell(\omega,b)} \cong \mathcal{Q}_{2 \sqrt{-1}\partial \omega}.
$$
\begin{proof}
We have 
\begin{equation*}
\ell(\omega,b)^\perp = e^{b + \sqrt{-1}\omega}(T^{1,0}) \oplus T^*_{1,0}\oplus \ell
\end{equation*}
and therefore there is a smooth bundle isomorphism
\begin{gather}\label{eq:QLstd}
\begin{split}
\mathcal{Q}_{\ell(\omega,b)} &\to T^{1,0} \oplus T_{1,0}^*\\
[e^{b + \sqrt{-1}\omega} Y + \eta]  &\mapsto Y + \eta.
\end{split}
\end{gather}
The agreement of the pairing and the anchor map with the ones on $\mathcal{Q}_{2 \sqrt{-1}\partial \omega}$ is straightforward. Let us now express the Dolbeault operator in terms of \eqref{eq:QLstd}. Given $X \in T^{0,1}$ and $Y + \eta  \in T^{1,0} \oplus T_{1,0}^*$, we have
	\begin{align*}
[e^{b + \sqrt{-1}\omega} & X, e^{b + \sqrt{-1}\omega} Y + \eta]\\
& =  e^{b + \sqrt{-1}\omega}[X,Y]^{1,0} + L_X \eta + i_Yi_X(H_ 0 + db + \sqrt{-1}d\omega)
& \textrm{mod} \ \ell\\
& = e^{b + \sqrt{-1}\omega} \overline\partial_X Y  + i_{X} \overline\partial \eta - i_X i_Y (2\sqrt{-1}\partial \omega), & \textrm{mod} \ \ell
\end{align*} 
which recovers the Dolbeault operator in Definition \ref{def:Q0} when $\tau = 2 \sqrt{-1}\partial \omega$. Similarly, for $X + \xi, Y+ \eta$ holomorphic sections of $T^{1,0} \oplus T_{1,0}^*$, we find
\begin{align*}
[e^{b + \sqrt{-1}\omega} X + \xi, e^{b + \sqrt{-1}\omega} Y + \eta]  & =  e^{b + \sqrt{-1}\omega}[X,Y] + L_X \eta - i_Y d\xi + i_Yi_X(2\sqrt{-1}\partial\omega)\\
& = e^{b + \sqrt{-1}\omega}[X,Y] + L_X \eta - i_Y d\xi  + i_Y \overline\partial \xi\\
& = e^{b + \sqrt{-1}\omega}[X,Y] + \partial(\eta(X)) + i_X \partial \eta  - i_X \partial \xi
\end{align*} 
as claimed.
\end{proof}
\end{lemma}

\begin{rmk}\label{rem:AeppliQiso}
With the notation of Lemma \ref{l:lifting} one has that $dd^c \omega = 0$ and hence a choice of lifting $\ell \subset (T \oplus T^*) \otimes \mathbb{C}$ yields a well-defined Aeppli class
$$
[\omega] \in H^{1,1}_A.
$$
In fact, if $\ell$ and $\ell'$ are liftings such that $[\omega] = [\omega'] \in H^{1,1}_A$, then it follows that $\mathcal{Q}_\ell \cong \mathcal{Q}_{\ell'}$. To see this, notice that
$$
\omega' = \omega + \overline{\partial} \xi^{1,0} + \partial \eta^{0,1}
$$
implies $\partial \omega' = \partial \omega - d \partial \xi^{1,0}$. Our claim follows combining Lemma \ref{l:liftingQ} with Theorem \ref{t:holCourant}.
\end{rmk}

\subsection{Generalized Hermitian metrics}

We introduce next a new ingredient, namely, generalized metrics, which will lead us naturally to the study of Hermitian metrics on exact holomorphic Courant algebroids. Recall that a generalized metric on a smooth exact Courant algebroid $E$ is given by an orthogonal decomposition 
$$
E = V_+ \oplus V_-
$$
such that the restriction of the ambient metric to $V_+$ (resp. $V_-$) is positive definite (resp. negative definite). Recall also that a generalized metric determines uniquely a Riemann metric $g$ on $M$ and an isotropic splitting of $E$. In particular, it has an associated isomorphism $E \cong (T\oplus T^*, \IP{,}, [,]_{H},\pi)$ for a uniquely determined closed three-form $H$, such that
\begin{equation}\label{eq:Vpm}
V_\pm = \{X \pm g(X), X \in T\}.
\end{equation}
The basic interaction between generalized metrics and complex geometry is provided by the following definition.

\begin{defn}\label{d:generalizedmetricomp}
Let $(M,J)$ be a complex manifold endowed with a smooth exact Courant algebroid $E$. We say that a generalized metric $E = V_+ \oplus V_-$ is \emph{compatible with $J$} if
$$
\ell = \{e \in V_+ \otimes \mathbb{C},\ \pi(e) \in T^{0,1}\} \subset E \otimes \mathbb{C}
$$
is a lifting of $T^{0,1}$.
\end{defn}

Using the splitting of $E$ determined by the generalized metric, it is not difficult to see that Definition \ref{d:generalizedmetricomp} implies that $g$ is Hermitian and furthermore
$$
\ell = e^{\sqrt{-1}\omega}T^{0,1}
$$
where $\omega = g J$ is the associated Hermitian form. Applying now Lemma \ref{l:lifting} we obtain the following.

\begin{lemma}\label{l:lifting2}
Let $(M,J)$ be a complex manifold endowed with a smooth exact Courant algebroid $E$. A generalized metric $E = V_+ \oplus V_-$ is compatible with $J$ if and only if the associated Riemannian metric $g$ is Hermitian and furthermore 
\begin{equation}\label{eq:liftingcondmet}
H = - d^c \omega.
\end{equation}
In particular $g$ is pluriclosed.
\end{lemma}

Given a compatible generalized metric, we can find an alternative presentation of the associated holomorphic Courant algebroid $\mathcal{Q}_\ell \cong \mathcal{Q}_{2 \sqrt{-1}\partial \omega}$ (see Lemma \ref{l:liftingQ}) which will naturally endow this bundle with a Hermitian metric. To see this, note that $V_+^\perp = V_-$ implies that
$$
\ell^\perp = (V_- \otimes \mathbb{C}) \oplus \ell.
$$
Therefore, as a smooth orthogonal bundle $\mathcal{Q}_\ell$ is canonically isomorphic to 
$$
\mathcal{Q}_\ell := \ell^\perp/\ell \cong V_- \otimes \mathbb{C}.
$$

\begin{defn}\label{d:generalizedmetriHerm}
Let $(M,J)$ be a complex manifold endowed with a smooth exact Courant algebroid $E$ and a compatible generalized metric $E = V_+ \oplus V_-$. Then, the induced \emph{generalized Hermitian metric} $G$ on $\mathcal{Q}_\ell$ is defined by
$$
G([s_1],[s_2]) = - 2\IP{\pi_- s_1, \overline{\pi_- s_2}}
$$
for $[s_j] \in \ell^\perp/\ell$ and $\pi_- \colon \ell^\perp \to V_- \otimes \mathbb{C}$ the orthogonal projection.
\end{defn}

We are ready to prove the main result of this section, where we calculate the Chern connection of the induced generalized Hermitian metric $G$ in terms of the connection $\nabla^-$ associated to the underlying pluriclosed structure (see \eqref{eq:Nablaminus}). This result provides an interpretation of \cite[Theorem 2.9]{Bismut} in the language of holomorphic Courant algebroids. 

\begin{prop}\label{t:Cherngeneralized} 
Let $(M,J)$ be a complex manifold endowed with a smooth exact Courant algebroid $E$ and a compatible generalized metric $E = V_+ \oplus V_-$. Then, the Chern connection of the associated generalized Hermitian metric $G$ on $\mathcal{Q}_\ell$ is given by
\begin{align}\label{eq:Nablaminusbisabs}
i_X \N^C_G s = \pi_-[\sigma_+ X,\pi_- s]
\end{align}
via the isomorphism $\mathcal{Q}_\ell \cong V_- \otimes \mathbb{C}$. Here, $\sigma_+ X = X + g(X)$ is the inverse of the isomorphism $\pi_{|V_+} \colon V_+ \to T$. More explcitly, via the identification $V_- \cong T$, the Chern connection is given by
\begin{align}\label{eq:Nablaminusbis}
\N^-_X Y = \N_X Y + \tfrac{1}{2} g^{-1}d^c \gw(X,Y,\cdot).
\end{align}
\begin{proof}
Observe that the right hand side of \eqref{eq:Nablaminusbisabs} defines an orthogonal connection on $V_-$, which can be identified with $\N^-$ via the isomorphism $\pi_{|V_-} \colon V_- \to T$ (see e.g. \cite[Proposition 3.14]{garciafernStreets2020}). Therefore, $\N^-$ extends $\mathbb{C}$-linearly to a $G$-unitary connection on $V_- \otimes \mathbb{C}$. By the abstract definition of the Dolbeault operator on $\mathcal{Q}_\ell$, we inmediately see that $(\nabla^-)^{0,1}$ coincides with $\overline{\partial}^\ell$.
\end{proof}
\end{prop}

In our next result we calculate an explicit formula for the generalized Hermitian metric $G$ in terms of the isomorphism $\mathcal{Q}_\ell \cong \mathcal{Q}_{2 \sqrt{-1}\partial \omega }$ in Lemma \ref{l:liftingQ}.

\begin{lemma}\label{t:Ggeneralized1} 
Let $(M,J)$ be a complex manifold endowed with a smooth exact Courant algebroid $E$ and a compatible generalized metric $E = V_+ \oplus V_-$. Then, the orthogonal isomorphism $\psi \colon \mathcal{Q}_{2 \sqrt{-1}\partial \omega} \to V_-\otimes \mathbb{C}$ induced by Lemma \ref{l:liftingQ} is given by
$$
\psi(X + \xi) = e^{\sqrt{-1}\omega}X - \tfrac{1}{2} e^{-\sqrt{-1}\omega} g^{-1}\xi.
$$
Consequently,
$$
\psi^*G(X + \xi,X + \xi) = 2g(X,\overline{X}) + (2g)^{-1}(\xi,\overline{\xi}).
$$
\begin{proof}
The first part follows from
$$
\psi(X + \xi) = e^{\sqrt{-1}\omega}X + \pi_-\xi = e^{\sqrt{-1}\omega}X + \tfrac{1}{2}(\xi - g^{-1}\xi) = e^{\sqrt{-1}\omega}X - \tfrac{1}{2} e^{-\sqrt{-1}\omega} g^{-1}\xi.
$$
The second part is a straightforward calculation and is left to the reader.
\end{proof}
\end{lemma}

For our applications it will be convenient to fix a background exact holomorphic Courant algebroid and generalized Hermitian metric. This motivates the following definition, which is inspired by \cite{garciafern2018canonical,garciafern2020gauge}.

\begin{defn}\label{def:metricQ}
Let $(M,J)$ be a complex manifold endowed with an exact holomorphic Courant algebroid $\mathcal{Q}$. A \emph{generalized Hermitian metric} on $\mathcal{Q}$ is given by a triple $(E,V_+,\varphi)$, where
\begin{enumerate}

\item $E$ is an exact Courant algebroid over $M$,

\item $V_+ \subset E$ is a generalized metric compatible with $J$,

\item $\varphi \colon \mathcal{Q}_\ell \to \mathcal{Q}$ is an isomorphism of holomorphic Courant algebroids inducing the identity on $M$.

\end{enumerate}
\end{defn}

Observe that a generalized Hermitian metric $(E,V_+,\varphi)$ on $\mathcal{Q}$ induces a generalized Hermitian metric $G'$ on $\mathcal{Q}_\ell$ as in Definition \ref{d:generalizedmetriHerm}. Therefore, via the isomorphism $\varphi$ we obtain a Hermitian metric
$$
G = \varphi_* G'
$$
on $\mathcal{Q}$ compatible with the orthogonal structure. By abuse of notation, we will also call $G$ a \emph{generalized Hermitian metric}. We next unravel the previous definition in terms of the model in Definition \ref{def:Q0}.

\begin{lemma} \label{l:SKTequivalence} 
Let $(M,J)$ be a complex manifold, and $\tau_0 \in \Lambda^{3,0 + 2,1}$, $d\tau_0 = 0$. Then, there is a one to one correspondence between the set of generalized Hermitian metrics on $\mathcal{Q}_{2 \sqrt{-1}\tau_0}$ and
	\begin{equation*}
	 \Big\{\omega + \beta \; | \; \omega >0 , \; d \beta = \tau_0 - \partial \omega  \Big\} \subset \Lambda^{1,1}_{\mathbb{R}} \oplus \Lambda^{2,0}.
\end{equation*}
Furthermore, the generalized Hermitian metric $G = \varphi_* G'$ is given by
$$
G(X + \xi,X + \xi) = 2g(X,\overline{X}) + (2g)^{-1}(\xi + 2 \sqrt{-1}\beta(X), \overline{\xi} - 2 \sqrt{-1}\overline{\beta(X)}).
$$
\begin{proof}
The pair $(E,V_+)$ determines $\omega > 0$ and an isomorphism $E \cong (T\oplus T^*, \IP{,}, [,]_H)$ for $H = - d^c\omega$. By Lemma \ref{l:liftingQ}, $\mathcal{Q}_{\ell} \cong \mathcal{Q}_{2\sqrt{-1}\partial \omega}$ and, by Theorem \ref{t:holCourant}, the isomorphism $\varphi \colon \mathcal{Q}_\ell \to \mathcal{Q}_{\tau_0}$ corresponds to
$$
\varphi = e^{-2 \sqrt{-1}\beta} \colon T^{1,0} \oplus T^*_{1,0} \to T^{1,0} \oplus T^*_{1,0}
$$
for $\beta \in \Lambda^{2,0}$ satisfying $2 \sqrt{-1} d \beta = 2 \sqrt{-1}\tau_0 - 2\sqrt{-1}\partial \omega$. The last part of the statement is now straightforward from Lemma \ref{t:Ggeneralized1}.  We refer to \cite{garciafern2020gauge} for further details.
\end{proof}
\end{lemma}

\begin{rmk}\label{rmk:twist}
For our applications to the pluriclosed flow, we will need to fix an initial pluriclosed Hermitian metric $g$ and a background pluriclosed metric $g'$. We will not require that the associated holomorphic Courant algebroids $\mathcal{Q}_{2\sqrt{-1}\partial \omega}$ and $\mathcal{Q}_{2\sqrt{-1}\partial \omega'}$ are isomorphic, but the weaker condition of being isomorphic as holomorphic orthogonal bundles. In practice, this boils down to the explicit condition   
\begin{equation}\label{eq:GHermitianexpweak}
\overline{\partial} \beta = \partial \omega' - \partial \omega
\end{equation}
for some $\beta \in \Lambda^{2,0}$ or, equivalently,
$$
[\partial \omega] = [\partial \omega'] \in H^{2,1}_{\overline{\partial}}(X).
$$
One can easily see that a pair $(\omega,\beta)$ as in \eqref{eq:GHermitianexpweak} defines a generalized Hermitian metric in the exact holomorphic Courant algebroid $\mathcal{Q}_{2\sqrt{-1}\partial \omega' + 2 \sqrt{-1}\partial \beta}$, that is, the twisting of $\mathcal{Q}_{2\sqrt{-1}\partial \omega'}$ by the $d$-closed $(3,0)$-form $2 \sqrt{-1}\partial \beta$ (see Definition \ref{def:Q0}).  
\end{rmk}

\subsection{Bismut's Identity from the classical viewpoint}\label{s:BismutId}

As originally observed in \cite[Theorem 2.9]{Bismut} in a different context, there is a striking relationship between the curvature of the Chern connection associated to the generalized Hermitian metric $G$ in Definition \ref{d:generalizedmetriHerm} and the curvature of the Bismut connection associated to the underlying pluriclosed structure. In the setup of generalized geometry this is a straightforward consequence of Proposition \ref{t:Cherngeneralized} combined with Proposition \ref{p:curvatureflip}. Since this relation plays an important role in our main results, in the present and the next section we discuss at length this interesting link from a classical point of view. In the sequel, the Dorfman bracket in the holomorphic Courant algebroids of our interest will play no role, and hence we will focus on the underlying holomorphic orthogonal bundle given by the sequence \eqref{eq:holCouseq}. We will also use a different normalization for $\partial \omega$ in the Dolbeault operator which simplifies the notation.

Let $(M,g,J)$ be a pluriclosed Hermitian manifold. The complexification $TM \otimes \mathbb{C}$ of the tangent bundle $TM$ is endowed with two natural geometric structures: a complex-valued metric induced by the $\mathbb{C}$-linear extension of $g$
$$
\IP{X,Y}_0 := - \tfrac{1}{2}g(X,Y)
$$
and a real structure, that is, a $\mathbb{C}$-antilinear involution $\sigma_0 \colon TM \otimes \mathbb{C} \to TM \otimes \mathbb{C}$, induced by complex conjugation
$$
\sigma_0(X) = \overline{X}.
$$
The involution $\sigma_0$ is orthogonal with respect to $\IP{,}_0$, and the combination of both yields the Hermitian metric
$$
G_0(X,Y) := - 2\IP{X,\sigma_0(Y)}_0 = g(X,\overline{Y}).
$$
Consider the isomorphism of smooth complex vector bundles $\psi_g \colon TM \otimes \mathbb{C} \to T^{1,0} \oplus T^*_{1,0}$ induced by the pluriclosed structure $(g,J)$, defined by
\begin{equation}\label{eq:Bismutisom}
\psi_g(X) = X^{1,0} - \sqrt{-1} \omega(X^{0,1},\cdot ) = X^{1,0} - g(X^{0,1}) .
\end{equation}
In the next result we give a formula for the structure $(\IP{,}_0,\sigma_0,G_0)$ transported to $T^{1,0} \oplus T^*_{1,0}$ via $\psi_g$.

\begin{lemma} \label{l:Psicxisom} 
The isomorphism \eqref{eq:Bismutisom} induces a complex-valued metric $\IP{,} = (\psi_g)_*\IP{,}_0$, an orthogonal real structure $\sigma = (\psi_g)_*\sigma_0$, and a Hermitian metric $G' = (\psi_g)_*G_0$ on $T^{1,0} \oplus T^*_{1,0}$, given by
$$
\IP{X + \xi, X + \xi} = \xi(X), \qquad \sigma (X + \xi) = - g^{-1}\overline{\xi} - g \overline{X}, 
$$
and
$$
G'(X + \xi,X+ \xi) := -2\IP{X+ \xi,\sigma(X+ \xi)} = g(X,\overline{X}) + g^{-1}(\xi,\overline{\xi}).
$$
In matrix notation, we have
\begin{align}\label{eq:G}
G' = \left( 
\begin{matrix}
g_{i \bj}  & 0 \\
0  & g^{\bl k}
\end{matrix}
\right).
\end{align}
\begin{proof}
We have $\psi_g^{-1}(X + \xi) = X - g^{-1} \xi$ where $g \colon T^*_{1,0} \to T^{0,1}$ is the isomorphism induced by the $\mathbb{C}$-linear extension of the metric $g$. The proof is now straightforward.
\end{proof}
\end{lemma}

Notice that the formula for the transported complex-valued metric $\IP{,}$ is independent of $g$ and coincides with the bilinear form of the holomorphic Courant algebroid $\mathcal{Q}_{\sqrt{-1}\partial \omega}$ in Definition \ref{def:Q0}. Furthermore, \eqref{eq:G} coincides up to normalization with the formula for the generalized Hermitian metric in Lemma \ref{t:Ggeneralized1}.  

In the next result we show that the holomorphic structure on $\mathcal{Q}_{\sqrt{-1}\partial \omega}$ has also a natural interpretation in terms of the complexified tangent bundle. For this, note that the $\mathbb{C}$-linear extension of the metric connection $\nabla^-$ to $TM \otimes \mathbb{C}$ is Hermitian (see \eqref{eq:Nablaminus}). Furthermore, using the fact that $R^B \in \Lambda^2 \otimes \Lambda^{1,1}$ combined with Proposition \ref{p:Bismutpair} it follows that 
$$
R^- \in \Lambda^{1,1} \otimes \End (TM \otimes \mathbb{C}),
$$ 
and therefore $\N^-$ induces a structure of holomorphic vector bundle on $TM \otimes \mathbb{C}$, with integrable Dolbeault operator $(\N^-)^{0,1}$.

\begin{lemma} \label{l:DolbeaultNabla-} 
The map $\psi_g \colon TM \otimes \mathbb{C} \to T^{1,0} \oplus T^*_{1,0}$ defined by \eqref{eq:Bismutisom} induces an isomorphism of orthogonal holomorphic vector bundles 
$$
\psi_g \colon ( TM \otimes \mathbb{C},(\N^-)^{0,1}) \to \QQ_{\sqrt{-1} \partial \omega}.
$$
\begin{proof}
By Lemma \ref{l:Psicxisom} we have that $\psi_g$ is a complex isometry. A straightforward calculation using the definitions of Bismut and Chern connections shows that
\begin{align*}
\IP{\N^B_X Y, Z} = \IP{\N^C_X Y, Z} + \tfrac{1}{2} d \gw (JX, Y, Z) + \tfrac{1}{2} d \gw(JX,JY,JZ),
\end{align*}
which in turn implies
$$
\nabla^-_{X^{0,1}}Y = \nabla^C_{X^{0,1}}Y - \sqrt{-1} g^{-1} \partial \omega ({X^{0,1}},Y^{0,1},\cdot).
$$
Therefore, using that the Chern connection is metric compatible, we obtain 
$$
\psi_g (\nabla^-_{X^{0,1}}\psi^{-1}_g(Y^{1,0} + \xi^{1,0})) = \nabla^C_{X^{0,1}}(Y^{1,0} + \xi^{1,0}) + \sqrt{-1} \partial \omega ({X^{0,1}},Y^{1,0},\cdot),
$$
which corresponds to the Dolbeault operator in Definition \ref{def:Q0} for $\tau = \sqrt{-1} \partial \omega$.
\end{proof}
\end{lemma}

In the next result we recover Proposition \ref{t:Cherngeneralized} in terms of classical geometry. This observation is originally due to Bismut \cite[Theorem 2.9]{Bismut}. 

\begin{prop}[Bismut's Identity] \label{p:Bismutid}
Let $(M,g,J)$ be a pluriclosed Hermitian manifold. Consider the associated holomorphic orthogonal bundle $\mathcal{Q}_{\sqrt{-1}\partial \omega}$, as in Definition \ref{def:Q0}, endowed with the Hermitian metric $G'$ (see \eqref{eq:G}). Then, the Chern connection $\N_{G'}^C$ of $G'$ satisfies
$$
\N^C_{G'} = (\psi_g)_*\nabla^-, \qquad  \Omega^C_{G'} = (\psi_g)_* R^-
$$
\begin{proof}
The statement follows from Lemma \ref{l:DolbeaultNabla-} combined with $\psi_* G_0 = G'$, and the uniqueness of the Chern connection on a Hermitian holomorphic vector bundle.
\end{proof}
\end{prop}

Similarly as in the proof of Lemma \ref{l:SKTequivalence}, given a background pluriclosed metric $g_0$ and a pair $(\omega,\beta)$ satisfying
\begin{equation}\label{eq:GHermitianexpweak2}
\overline{\partial} \beta = \partial \omega_0 - \partial \omega
\end{equation}
one has an induced isomorphism of holomorphic orthogonal bundles 
$$
e^{-\sqrt{-1}\beta} \colon\mathcal{Q}_{\sqrt{-1}\partial \omega} \to \mathcal{Q}_{\sqrt{-1}\partial \omega_0}.
$$
Thus, $(\omega,\beta)$ induces in a natural way a Hermitian metric on $\mathcal{Q}_{\sqrt{-1}\partial \omega_0}$ compatible with the orthogonal structure, defined by $G = (e^{\sqrt{-1}\beta})^*G'$ (see Remark \ref{rmk:twist}). We close this section with a formula for the generalized Hermitian metric $G$ and its Chern connection. In particular, the next lemma combined with Proposition \ref{p:Bismutid} recovers the calculation in \cite[Lemma 2.1]{JordanStreets} in a more conceptual way. The proof is straightforward from Proposition \ref{p:Bismutid}.

\begin{prop} \label{p:GQ0}
Let $(M,g_0,J)$ be a pluriclosed Hermitian manifold. Consider the associated orthogonal holomorphic vector bundle $\mathcal{Q} = \mathcal{Q}_{\sqrt{-1}\partial \omega_0}$, as in Definition \ref{def:Q0}. Given $(\omega,\beta)$ as in \eqref{eq:GHermitianexpweak2}, the induced generalized Hermitian metric $G = (e^{\sqrt{-1}\beta})^*G'$ on $\mathcal{Q}$ is given by
$$
G(X + \xi,X + \xi) = g(X,\overline{X}) + g(\xi + \sqrt{-1}\beta(X), \overline{\xi}  - \sqrt{-1} \overline{\beta(X)}).
$$
In matrix notation, we have
\begin{align}\label{eq:Gbeta}
G = \left( 
\begin{matrix}
g_{i \bj} + \gb_{i k} \bgb_{\bj \bl} g^{\bl k} & \sqrt{-1}\gb_{ip} g^{\bl p}\\
- \sqrt{-1} \bgb_{\bj \bp} g^{\bp k}  & g^{\bl k}
\end{matrix}
\right).
\end{align}
Furthermore, the Chern connection $\N_{G}^C$ of $G$ satisfies
$$
(e^{\sqrt{-1}\beta})_*\N^C_G = (\psi_g)_*\nabla^-, \qquad (e^{\sqrt{-1}\beta})_* \Omega^C_G = (\psi_g)_* R^-.
$$
\end{prop}

\subsection{Explicit formulas}

Proposition \ref{p:GQ0}, in conjunction with Proposition \ref{p:curvatureflip}, gives a complete description of the Chern connection of the generalized metric $G$ in terms of the Bismut curvature of $g$.  In this subsection we give a more direct and explicit proof of Proposition \ref{p:GQ0}, extending the approach of \cite[Lemma 2.1]{JordanStreets}.  These explicit identities provide some further insight into the relationship of these curvatures, and play a role in deriving certain analytic estimates for pluriclosed flow below.

To start, we derive a basic relationship between the Chern and Bismut curvatures of an arbitrary Hermitian metric.

\begin{lemma}\label{l:Bismut11Herm}
	Given $(M^{2n}, g, J)$ a Hermitian manifold, then in any local complex coordinate chart one has
	\begin{align*}
		\delb T_{ij\bk\bl}=&\ -(\Omega^C)_{i\bk j\bl}+(\Omega^B)_{j\bl i \bk}+T_{j\ga\bk}g^{\bgb \ga}\bar{T}_{\bl \bgb i},
	\end{align*}
where $T = -\i \del \gw$ is the torsion of the Chern connection.
	
	\begin{proof} Fix some local complex coordinates, and let $\gG$ denote the associated Chern connection coefficients.  It follows from Definition \ref{d:Chernconn} that
	\begin{align*}
	\N^B_i\del_j =&\ \Gamma_{ji}^l \del_l, \qquad \N_{\bi}^B\del_j = \bar{T}_{\bi \bk j}g^{\bk l}\del_l.
	\end{align*} 	
		Using this, we derive and identity relating the $\Lambda^{1,1}\otimes \Lambda^{1,1}$ part of the Bismut curvature and the Chern curvature:
		\begin{align*}
		(\Omega^B)_{i\bj k \bl} =&\ g(\N^B_i {\N}_{\bj}^B \del_k - {\N}_{\bj}^B\N^B_i \del_k,\delb_l)\\
		=&\ (\Omega^C)_{k\bj i \bl} + \N^C_i \bar{T}_{\bj \bl k}+T_{ik}^p\bar{T}_{\bj\bl p}-T_{ip\bl}g^{\bq p}\bar{T}_{\bj \bq k}\\
		=&\ (\Omega^C)_{k\bj i \bl} + (\Omega^C)_{i\bl k\bj}-(\Omega^C)_{i\bj k\bl}+T_{ik}^p\bar{T}_{\bj\bl p}-T_{ip\bl}g^{\bq p}\bar{T}_{\bj \bq k}.
		\end{align*}
Using this we furthermore obtain
		\begin{align*}
		\delb T_{ij\bk\bl}=&\ \delb_k T_{ij \bl} - \delb_l T_{ij\bk}\\
		=&\ g_{j\bl,i\bk}-g_{i\bl,j\bk} - g_{j\bk,i\bl}+g_{i\bk,j\bl}\\
		=&\ -(\Omega^C)_{i\bk j\bl}+(\Omega^C)_{j\bk i\bl}+(\Omega^C)_{i\bl j\bk} - (\Omega^C)_{j\bl i\bk}\\
		&\ +g^{\bq p}g_{j\bq,i}g_{p\bl,\bk}-g^{\bq p}g_{i\bq,j}g_{p\bl,\bk}-g^{\bq p}g_{j\bq,i}g_{p\bk,\bl}+g^{\bq p}g_{i\bq,j}g_{p\bk,\bl}\\
		=&\ -(\Omega^C)_{i\bk j\bl}+(\Omega^C)_{j\bk i\bl}+(\Omega^C)_{i\bl j\bk} - (\Omega^C)_{j\bl i\bk} +\bar{T}_{\bk\bl}^{\bq}T_{ij\bq}\\
		=&\ -(\Omega^C)_{i\bk j\bl}+(\Omega^B)_{j\bl i \bk}-T_{ji}^\ga\bar{T}_{\bl\bk \ga}+T_{j\ga\bk}g^{\bgb \ga}\bar{T}_{\bl \bgb i} +\bar{T}_{\bk\bl}^{\bq}T_{ij\bq}\\
		=&\ -(\Omega^C)_{i\bk j\bl}+(\Omega^B)_{j\bl i \bk}+T_{j\ga\bk}g^{\bgb \ga}\bar{T}_{\bl \bgb i},
		\end{align*}
		as required.
	\end{proof}	
\end{lemma}

\begin{prop}\label{p:Bismut11PC}
	Given $(M^{2n},g,J)$ pluriclosed, one has
	$$(\Omega^B)_{i\bj k\bl} = (\Omega^C)_{k\bl i\bj} - T_{ip\bl}g^{\bq p}\bar{T}_{\bj\bq k}.$$
\begin{proof} This follows from Lemma \ref{l:Bismut11Herm}, using that $\delb T = 0$ since $g$ is pluriclosed.
\end{proof}
\end{prop}

Next we compute the $\Lambda^{2,0}\otimes \Lambda^{1,1}$ component of the Bismut curvature.

\begin{prop}\label{p:Bismut20PC} Given
	$(M^{2n},g,J)$ pluriclosed, one has
	$$(\Omega^B)_{ijk\bl}=\N^C_kT_{ij\bl}.$$
	\begin{proof} We first note that as a consequence of the Bianchi identities for a general Hermitian metric one has
	$$(\Omega^B)_{ijk\bl} = \N^C_iT_{kj\bl} - \N^C_jT_{ki\bl} + T_{ij}^{\gl}T_{k\gl \bl} - T_{ik}^{\gl}T_{j\gl\bl}+T_{jk}^{\gl}T_{i\gl\bl}.$$
	Since we know that the Chern torsion is given by $T = -\i \del \omega$, we have the identity
	$$ \del T_{ijk\bl}=0.$$
	This can be rewritten in terms of the Chern connection as
	$$\N^C_i T_{jk\bl}-\N^C_jT_{ik\bl}+\N^C_kT_{ij\bl}+T_{ij}^pT_{pk\bl}+T_{jk}^pT_{pi\bl}-T_{ik}^pT_{pj\bl} = 0.$$
	Combining this with the first equation of the proof then yields the claim.
	\end{proof}
\end{prop}

Next we give an explicit computation of the Chern connection for the generalized Hermitian metric associated to a given pluriclosed metric. The following result gives an alternative proof of the identification of connections from Proposition \ref{p:GQ0}.  Furthermore, we will use these formulae in relating analytic estimates for $G$ and $g$ along pluriclosed flow below.  In the statement below we choose $\{z^i\}$ local complex coordinates with associated coordinate basis $\{\frac{\del}{\del z^i} \}$ and cobasis $\{dz^i\}$.  Together these form a basis for $T^{1,0} \oplus T^*_{1,0}$, referred to as
\begin{align*}
Z^i = \frac{\del}{\del z^i}, \qquad W^i = dz^i.
\end{align*}

\begin{lemma}\label{l:GChernCxn}
Let $(M,g_0,J)$ be a pluriclosed manifold. Consider the associated orthogonal holomorphic vector bundle $\QQ = \mathcal Q_{\i \del \gw_0}$ as in Definition \ref{def:Q0}, and a generalized Hermitian metric $G = G(\gw, \gb)$ as in Proposition \ref{p:GQ0}.  Then the Chern connection $\N^C_G$ of $G$ is expressed in local complex coordinates as
	\begin{align*}
	\Gamma_{i Z^a}^{Z^b} =&\ g^{\bc b}g_{a\bc,i} - \sqrt{-1} g^{\bc b}g^{\bq p}\gb_{ap}\bar{T}_{\bq\bc i},\\
	\Gamma_{i Z^a}^{W^b} =&\ \N^{C}_i\gb_{ab} + g^{\bq p}g^{\bc m}\gb_{ap}\gb_{bm}\bar{T}_{\bq\bc i},\\
	\Gamma_{i W^a}^{Z^b} =&\ -g^{\bc b}g^{\bq a}\bar{T}_{\bq\bc i},\\
	\Gamma_{i W^a}^{W^b} =&\ -g^{\bq a}g_{b\bq,i} - \sqrt{-1}g^{\bq a}g^{\bc m}\gb_{bm}\bar{T}_{\bq\bc i}.
	\end{align*}
	\begin{proof}
Choose some local coordinates $\{z^i\}$ on $M$, with associated basis $\{e_A\} = \{\frac{\del}{\del z^i}, dz^i\}$.  Before we begin computing, we first define a section of $T^*_{0,1} \otimes \End(T^{1,0} \oplus T^*_{1,0} )$ by the following.
	$$i_{\delb_i}\mathcal{T}^0(X+\xi) = \begin{pmatrix}
	0 &\ 0\\
	T^0(\cdot,\delb_i,\cdot) &\ 0
	\end{pmatrix}\begin{pmatrix}X\\
	\xi\end{pmatrix} = T^0(X,\delb_i,\cdot).$$
	Using this, we can write the twisted-$\delb$ of an arbitrary section $s$ of $T^{1,0}M\oplus T^*_{1,0} $ in the following way (denoting the usual $\delb$ by $\delb^0$),
	$$\delb s = \delb^0 s + \mathcal{T}^0(s).$$
Thus we apply metric compatibility of the Chern connection to derive
		\begin{align*}
		\del_i G(e_A,\bar{e}_B)=&\ G((\N^C_G)_ie_A,\bar{e}_B) + G(e_A, i_{\del_i}\del \bar{e}_B)\\
		=&\ \Gamma_{iA}^PG_{P\bar{B}} + G(e_A, i_{\del_i}\bar{\mathcal{T}}^0(\bar{e}_B)) \\
		=&\ \Gamma_{iA}^PG_{P\bar{B}} + (\bar{\mathcal{T}}^0)_{i\bar{B}}^{\bar{P}} G_{A\bar{P}}.
		\end{align*}
	Rearranging gives
	$$\Gamma_{iA}^B = G^{\bar{P}B}(G_{A\bar{P},i} - (\bar{\mathcal{T}}^0)_{i\bar{P}}^{\bar{Q}}G_{A\bar{Q}}),$$
	which, together with the explicit form of $G$, yields the lemma.
	\end{proof}
\end{lemma}

\begin{prop} \label{p:GChernCurvature}
Let $(M,g_0,J)$ be a pluriclosed manifold. Consider the associated orthogonal holomorphic vector bundle $\QQ = \mathcal Q_{\i \del \gw_0}$ as in Definition \ref{def:Q0}, and a generalized Hermitian metric $G = G(\gw, \gb)$ as in Proposition \ref{p:GQ0}.  Then the Chern curvature $\Omega^C_G$ of the Hermitian metric $G$ has coefficients
\begin{align*}
e^{-\i \gb} (\Omega^C_G)_{i\bj}(e^{ \sqrt{-1}\gb} \del_k) =&\ (\Omega^B)_{k\bm i\bj}g^{\bm l}\del_l - (\Omega^B)_{kli\bj}dz^l,\\
e^{-\i \gb} (\Omega^C_G)_{i\bj}(e^{ \sqrt{-1}\gb} dz^k) =&\ g^{\bc l}g^{\bq k}\Omega^B_{\bq\bc\bj i}\del_l - g^{\bq k}(\Omega^B)_{l\bq i\bj}dz^l.
\end{align*}

\begin{proof}
Fix $p \in M$. Since the computation is local without loss of generality we can conjugate by a locally defined holomorphic $(2,0)$ form to assume that $\beta(p) = 0$.  We begin by computing the curvature operator $\Omega^G := \Omega^C_G$ acting on sections of the tangent bundle. The following computations make use of Lemma \ref{l:GChernCxn} which gives the connection coefficients of $\N^G : = \N^C_G$ in holomorphic coordinates:
\begin{align*}
(\Omega^G)_{i\bj}\del_k =&\ \N^G_i\N^G_{\bj}\del_k - \N^G_{\bj}\N^G_{i}\del_k\\
=&\ -\N^G_i(T^0_{kl\bj}dz^l) - \N^G_{\bj}(\Gamma_{iZ^k}^{Z^l}\del_l + \Gamma_{iZ^k}^{W^l}dz^l)\\
=&\ -T^0_{kl\bj,i}dz^l - T^0_{kp\bj}\Gamma_{iW^p}^{Z^l}\del_l - T^0_{kp\bj}\Gamma_{iW^p}^{W^l}dz^l - \Gamma_{iZ^k,\bj}^{Z^l}\del_l + \Gamma_{iZ^k}^{Z^p}T^0_{pl\bj }dz^l-\Gamma_{iZ^k,\bj}^{W^l}dz^l\\
=&\ [-T^0_{kp\bj}\Gamma_{iW^p}^{Z^l} - \Gamma_{iZ^k,\bj}^{Z^l} ] \del_l + [-T^0_{kl\bj,i} - T^0_{kp\bj}\Gamma_{iW^p}^{W^l} + \Gamma_{iZ^k}^{Z^p}T^0_{pl\bj } - \Gamma_{iZ^k,\bj}^{W^l}] dz^l\\
=&\ [(\Omega^C)_{i\bj k}^l - \sqrt{-1} \gb_{kr,\bj}g^{\bs r}g^{\bp l}\bar{T}_{\bp \bs i} - g^{\bs l}g^{\br p}\bar{T}_{\bs \br i}T^0_{kp\bj}]\del_l - [\N_i^C T_{kl\bj}]dz^l\\
=&\ [(\Omega^C)_{i\bj k}^l - g^{\bs r}g^{\bp l}\bar{T}_{\bp \bs i}(\sqrt{-1} \gb_{kr,\bj} + T^0_{kr\bj})]\del_l - [\N_i^C T_{kl\bj}]dz^l\\
=&\ [(\Omega^C)_{i\bj k}^l - g^{\bs r}g^{\bp l}\bar{T}_{\bp \bs i}T_{kr\bj}]\del_l - [\N_i^C T_{kl\bj}]dz^l\\
=&\ (\Omega^B)_{k\bm i\bj}g^{\bm l}\del_l - (\Omega^B)_{kli\bj}dz^l,
\end{align*}
where the last line follows from Propositions \ref{p:Bismut11PC} and \ref{p:Bismut20PC}.  We proceed similarly for the curvature operator acting on sections of the cotangent bundle.
\begin{align*}
(\Omega^G)_{i\bj}dz^k =&\ -\N^G_{\bj} (\Gamma_{iW^k}^{Z^l}\del_l + \Gamma_{iW^k}^{W^l}dz^l)\\
=&\ -\Gamma_{i W^k,\bj}^{Z^l}\del_l - \Gamma_{iW^k}^{Z^l}\N^G_{\bj}\del_l - \Gamma_{iW^k,\bj}^{W^l}dz^l\\
=&\ [-\Gamma_{i W^k,\bj}^{Z^l}]\del_l +[\Gamma_{iW^k}^{Z^p}T^0_{pl\bj} - \Gamma_{iW^k,\bj}^{W^l}]dz^l\\
=&\ [g^{\bc l}g^{\bq k}\N_{\bj}^{C}\bar{T}_{\bq\bc i}]\del_l + [- g^{\bc p}g^{\bq k}\bar{T}_{\bq \bc i}T^0_{pl\bj} - (\Omega^{C})_{i\bj l}^k + \sqrt{-1}g^{\bq k}g^{\bc p}\delb \gb_{l p\bj}\bar{T}_{\bq \bc i}]dz^l\\
=&\ [g^{\bc l}g^{\bq k}\N_{\bj}^{C}\bar{T}_{\bq\bc i}]\del_l + [-(\Omega^{C})_{i\bj l}^k + g^{\bq k}g^{\bc p}\bar{T}_{\bc\bq i} (\sqrt{-1}\delb\gb_{lp\bj}+T^0_{lp\bj})]dz^l\\
=&\ [g^{\bc l}g^{\bq k}\N_{\bj}^{C}\bar{T}_{\bq\bc i}]\del_l -[ (\Omega^{C})_{i\bj l}^k - g^{\bq k}g^{\bc p}\bar{T}_{\bc\bq i} T_{pl\bj}]dz^l\\
=&\ g^{\bc l}g^{\bq k}\Omega^B_{\bq\bc\bj i}\del_l - g^{\bq k}(\Omega^B)_{l\bq i\bj}dz^l,
\end{align*}
where again we applied the results of Propositions \ref{p:Bismut11PC} and \ref{p:Bismut20PC}.
\end{proof}
\end{prop}

\section{Slope stability criteria for Bismut Hermitian-Einstein metrics} \label{s:slopestability}

In this section we discuss necessary conditions for the existence of a Bismut Hermitian-Einstein metric
in relation to slope stability for the associated holomorphic Courant algebroid. First in \S \ref{ss:HYM} we recall fundamental points of Hermitian-Einstein metrics on holomorphic vector bundles.  Then in \S \ref{ss:BHE} we explain how Bismut Hermitian-Einstein metrics are in fact Hermitian-Einstein metrics on the relevant holomorphic Courant algebroid, leading to slope stability obstructions to existence.  Finally in \S \ref{ss:BHEobs} we derive concrete obstructions in terms of holomorphic maps from the underlying complex manifold, yielding explicit families in all dimensions of manifolds with vanishing first Chern class but no Bismut Hermitian-Einstein metric.

\subsection{Hermitian-Einstein metrics and slope stability} \label{ss:HYM}

\begin{defn} \label{d:HEdef} Let $(M^{2n}, g, J)$ be a Hermitian manifold and suppose $(\EE, h) \to M$ denote a holomorphic vector bundle with Hermitian metric $h$ and associated Chern connection $\N^h$. The \emph{second Ricci curvature} is
\begin{align*}
S^h_g := \sqrt{-1}\tr_{\gw} \Omega^h \in \End(E).
\end{align*}
\end{defn} 

The existence of metrics with vanishing second Ricci curvature curvature, or more generally Hermitian-Einstein metrics, is goverened by slope stability criteria as in the Donaldson-Uhlenbeck-Yau Theorem \cite{DonaldsonHYM, UYau} and its extensions to Hermitian manifolds (see \cite{BuchdahlHK,LiYauHYM,lt}).  To state the precise result which we will use, let us recall first some basic definitions. Given a coherent sheaf $\mathcal{F}$ of $\mathcal{O}_M$-modules over $M$, the determinant $\det \mathcal F : = ((\Lambda^r \mathcal F)^*)^*$, where $r$ denotes the rank of $\mathcal{F}$, is a holomorphic line bundle over $M$. Given now an Aeppli class $a \in H_A^{n-1,n-1}$, we can define the slope of $\mathcal{F}$ by
$$
\mu_{a}(\mathcal{F}) = \frac{c_1(\det \mathcal{F})\cdot a}{r}.
$$ 
where $c_1(\det \mathcal{F}) \in H^{1,1}_{BC}(M)$ is the first Chern class of $\det \mathcal{F}$, regarded an element in the Bott-Chern cohomology of $M$. Here we use the standard duality pairing
$$
H^{1,1}_{BC} \otimes H_A^{n-1,n-1} \to \mathbb{C}.
$$

\begin{defn} \label{d:stab} 
Let $(M,J)$ be a compact complex manifold endowed with an Aeppli class $a \in H_A^{n-1,n-1}$. A holomorphic vector bundle $\EE$ over $M$ is \emph{$a$-stable} if for any subsheaf $\mathcal{F} \subset \mathcal{\EE}$ one has
$$
\mu_{a}(\mathcal{F}) < \mu_{a}(\mathcal{Q}).
$$
We say that $\mathcal{E}$ is \emph{$a$-polystable} if it is a direct sum of $a$-stable bundles with the same slope. 
\end{defn}

The relation between slope stability and the Hermitian-Einstein equation is provided by the following important result (cf \cite{lt}):

\begin{thm}\label{t:DUY} Let $(M,J)$ be a compact complex manifold. Let $\tilde \omega$ be a Gauduchon Hermitian metric on $M$ with Aeppli class $a = [\tilde \omega^{n-1}] \in H_A^{n-1,n-1}$. A holomorphic vector bundle $\mathcal{\EE}$ over $(M,J)$ admits a Hermitian metric $h$ solving the Hermitian-Einstein equation
$$
S^h_g = 2\pi \mu_{a}(\mathcal{E}) \Id_\mathcal{E}
$$
if and only if $\mathcal{\EE}$ is $a$-polystable.
\end{thm}

\subsection{Bismut Hermitian-Einstein metrics and slope stability of \texorpdfstring{$\QQ$}{Q}} \label{ss:BHE}

The theory of Hermitian-Einstein metrics implies obstructions to the existence of Bismut Hermitian-Einstein metrics.  The fundamental point in our discussion is that the Hermitian Yang-Mills curvature of a generalized metric $G$ on $\QQ_{\i\del \gw}$ is completely determined by the Bismut Ricci curvature of the associated Hermitian metric $g$ on $M$.  We make this concrete below.

\begin{prop} \label{p:SecondGChernCurvature} Let $(M,g_0,J)$ be a pluriclosed manifold. Consider the associated orthogonal holomorphic vector bundle $\QQ = \mathcal Q_{\i \del \gw_0}$ as in Definition \ref{def:Q0}.  Consider a generalized Hermitian metric $G = G(\gw, \gb)$ as in Proposition \ref{p:GQ0}.  Then the Chern connection $\N^C_G$ of the Hermitian metric $G$, defined as in \eqref{eq:Gbeta} satisfies
\begin{align*}
	S^G_g =&\ \sqrt{-1}(e^{\sqrt{-1}\beta})^*\left( 
	\begin{matrix}
	 - g^{-1} \rho_B^{1,1} & g^{-1} \rho_B^{0,2} g^{-1}\\
	 \rho_B^{2,0} & - \rho_B^{1,1} g^{-1}
	\end{matrix} \right).
	\end{align*}
Consequently, $g$ is Bismut Hermitian-Einstein if and only if $G$ is Hermitian-Einstein with respect to $g$, that is, if and only if 
$$
S^G_g = 0.
$$	
	\begin{proof}
		This follows by taking the trace of the formula for $\Omega^C_G$ in Proposition \ref{p:GChernCurvature}.  Alternatively, this follows from  Proposition \ref{p:Bismutid}, since
\begin{align*}
S^G_g & = \sqrt{-1} \tr_\omega \Omega_G^C\\ 
& = \sqrt{-1}(e^{\sqrt{-1}\beta})^* \psi_* \tr_\omega R^-\\
& = - \sqrt{-1}(e^{\sqrt{-1}\beta})^* \psi_* g^{-1} \rho_B,
\end{align*}
where for the last identity we have used Proposition \ref{p:Bismutpair} to conclude
$$
(\tr_\omega R^-) X = \tfrac{1}{2} \sum_{j=1}^{2n}R^-(e_i,Je_i)X = - \tfrac{1}{2} \sum_{j=1}^{2n} g^{-1}R^B(X,\cdot,Je_i,e_i) = - g^{-1}\rho_B(X).
$$
Finally, given $X + \xi \in T^{1,0} \oplus T_{1,0}^*$, we calculate
\begin{align*}
(\psi_* g^{-1} \rho_B)(X) & = \psi(g^{-1}\rho_B(X)) = g^{-1}\rho_B^{1,1}(X) - \rho_B^{2,0}(X)\\
(\psi_* g^{-1} \rho_B)(\xi) & = - \psi(g^{-1}\rho_B g^{-1} \xi) = \rho_B^{1,1}g^{-1} \xi - g^{-1}\rho_B^{0,2} g^{-1} \xi.
\end{align*}
\end{proof}
\end{prop}

As a straightforward consequence of Theorem \ref{t:DUY} and Proposition \ref{p:SecondGChernCurvature} we obtain the following necessary condition to the existence of a pluriclosed Bismut Hermitian-Einstein metric.

\begin{cor} \label{cor:stab} 
Let $(M,g,J)$ be a pluriclosed Hermitian manifold. Denote by $a = [\tilde \omega^{n-1}] \in H_A^{n-1,n-1}$ the Aeppli class of the unique Gauduchon metric $\tilde{\omega}$ in the conformal class of $\omega$, such that $\int_X \tilde\omega^n = \int_X \omega^n$. Consider the associated orthogonal holomorphic vector bundle $\mathcal{Q} = \mathcal{Q}_{\sqrt{-1}\partial \omega}$ as in Definition \ref{def:Q0}. Assume that the metric $g$ is Bismut Hermitian-Einstein. Then, for any subsheaf $\mathcal{F} \subset \mathcal{Q}$ one has
\begin{equation}\label{eq:slopeineq}
\mu_{a}(\mathcal{F}) \leq 0,
\end{equation}
with equality only if $\mathcal{Q}$ splits holomorphically. 

\begin{proof}
The proof is an easy consequence of Proposition \ref{p:SecondGChernCurvature} and Theorem \ref{t:DUY}, after noting that $\mathcal{Q}$ satisfies $c_1(\mathcal{Q}) = 0 \in H^{1,1}_{BC}(M)$, since $\det \mathcal{Q}$ admits a canonical holomorphic trivialization induced by the holomorphic pairing.
\end{proof}
\end{cor}

\begin{rmk}
The holomorphic Courant algebroid $\mathcal{Q}$ has a natural orthogonal structure. Thus, by general theory, in the previous result we can restrict to isotropic subsheaves, that is $\mathcal{F} \subset \QQ$ such that $\IP{\mathcal{F},\mathcal{F}} = 0$, in order to measure the slope inequality \eqref{eq:slopeineq} (see \cite{BiswasGomez}).
\end{rmk}

As observed in Remark \ref{rem:AeppliQiso}, the holomorphic structure on $\mathcal{Q}_{\sqrt{-1}\partial \omega}$  is determined up to isomorphism by the Aeppli class of the pluriclosed metric $[\omega] \in H^{1,1}_A$. In general, the stability condition depends in a intricate way on the Bismut Hermitian-Einstein pluriclosed metric. This is due to the fact that the map 
$$
\omega \mapsto a = [\tilde \omega^{n-1}] \in H_A^{n-1,n-1}(M)
$$
is typically a complicated function in the space of pluriclosed metrics. In the special case of complex surfaces $n=2$, this map only depends on the Aeppli class $[\omega] \in H_A^{1,1}$, and is just the identity map.

\subsection{Obstructions to Bismut Hermitian-Einstein metrics} \label{ss:BHEobs}

In the following we will see that we can obtain some more concrete implications of Corollary \ref{cor:stab} for the existence of  pluriclosed Bismut Hermitian-Einstein metrics. In particular, our next result provides a clean obstruction to the existence of such metrics on a compact complex manifold. For this, we exploit the fact that any exact holomorphic Courant algebroid has a canonical isotropic subsheaf, given by the holomorphic cotangent bundle
\begin{align*}
T^*_{1,0} \overset{\pi^*}{\longrightarrow} \mathcal{Q}.
\end{align*}
We will say that an Aeppli class $a \in H_A^{n-1,n-1}(M)$ is \emph{positive} if $a = [\tilde \omega^{n-1}]$, for some Gauduchon metric $\tilde \omega$ on $M$.

\begin{thm} \label{thm:obstruction} 

Let $(M,J)$ be a compact connected complex manifold. Assume that $M$ admits a pluriclosed Hermitian metric $g$ which is Bismut Hermitian-Einstein. Then, there exists a positive Aeppli class $a \in H_A^{n-1,n-1}(M)$ such that, for any complex manifold $Z$ and any holomorphic map $f \colon M \to Z$ such that $df$ is surjective at one point, one has
$$
f^*c_1(Z) \cdot a \geq 0
$$ 
Furthermore, for $Z = (M,J)$ and $f = \Id$, $c_1(M)\cdot a > 0$ unless $g$ is K\"ahler.

\begin{proof}
Let $g$ as in the statement and consider the holomorphic Courant algebroid $\mathcal{Q} := \mathcal{Q}_{\sqrt{-1}\partial \omega}$. Let $a = [\tilde \omega^{n-1}]$ be the Aeppli class of the associated (normalized) Gauduchon metric $\tilde \omega$. By Corollary \ref{cor:stab}, the holomorphic vector bundle underlying $\mathcal{Q}$ is $a$-polystable. Let $f \colon M \to Z$ be as in the statement. Then, the differential $df \colon T_{1,0} \to T^{1,0}Z$ induces a morphism
\begin{align*}
f^*T^*_{1,0}Z \longrightarrow T^*_{1,0} \overset{\pi^*}{\longrightarrow} \mathcal{Q}.
\end{align*}
Since $df$ is surjective at one point, by Sard's Theorem there exists a dense open subset 
$R \subset Z$ of regular values of $f$. Using that $M$ is connected and that $f$ is holomorphic, it follows that $f^{-1}(R) \subset M$ is also dense (in fact, the complement is an analytic subspace of codimension $\geq 1$). Hence, the previous morphism induces a subsheaf
$$
f^*T^*_{1,0}Z \hookrightarrow \mathcal{Q}.
$$
The inequality of slopes \eqref{eq:slopeineq} now gives 
$$
0 \geq c_1(f^*T^*Z)\cdot a = - f^* c_1(Z) \cdot a.
$$
As for the last part, if $c_1(M)\cdot a = 0$ then by Proposition \ref{p:SecondGChernCurvature} and Corollary \ref{cor:stab} we have a holomorphic splitting $\mathcal{Q} = T^{1,0} \oplus T^*_{1,0}$ and hence $\partial \omega = 0$.
\end{proof}
\end{thm}

\begin{rmk}
An alternative proof of the last part of Theorem \ref{thm:obstruction} can be obtained as a consequence of a result of Gauduchon \cite{Gauduchonfibres}. For instance, $\rho_B = 0$ implies that the Chern scalar is non-negative and bigger than zero at some point if $g$ is non-K\"ahler (see Proposition \ref{p:BismutRicci}). From this, it follows that $c_1(M) \cdot a \geq 0$ with equality only if $g$ is K\"ahler (see  \cite[Proposition 8.30]{garciafernStreets2020}).
\end{rmk}

We next obtain a more concrete criteria derived from Theorem \ref{thm:obstruction}.

\begin{cor} \label{cor:obstructionfib}

Let $f \colon (M,J) \to Z$ be a holomorphic map of compact connected complex manifolds. Assume that $df$ is surjective at some point and that $Z$ is K\"ahler with $c_1(Z) < 0$. Then $M$ does not admit a Bismut Hermitian-Einstein pluriclosed metric.

\begin{proof}
By Aubin-Yau's Theorem there exists a K\"ahler-Einstein metric $\omega_Z$ on $Z$ with negative scalar curvature, that is, such that $\rho_Z = - \omega_Z$. Let $a = [\tilde \omega^{n-1}] \in H_A^{n-1,n-1}(M)$ be a positive Aeppli class on $M$. Then
$$
f^*c_1(Z) \cdot a = - \int_X f^* \omega_Z \wedge \tilde \omega^{n-1}.
$$ 
By hypothesis there exists $x \in M$ such that $f$ is a submersion, and hence, arguing as in the proof of Theorem \ref{thm:obstruction}, the preimage of the set of regular values is open and dense. On this locus $f^* \omega_Z \wedge \tilde \omega^{n-1} > 0$, and hence $f^*c_1(Z) \cdot a < 0$.
\end{proof}
\end{cor}

As a consequence of Corollary \ref{cor:obstructionfib}  we obtain examples of compact pluriclosed manifolds $(M^{2n},J)$ with $c_1(M) = 0 \in H^2(M,\mathbb{Z})$ which do not admit a Bismut Hermitian-Einstein metric. To the knowledge of the authors, this is the first class of such examples in the literature for dimension $n \geq 3$ (the case $n=2$ is settled by \cite{GauduchonIvanov}).  In order to present our examples we start with some general discussion of principal bundles over complex manifolds. Let $Z$ be a K\"ahler manifold. Let $T = \mathbb{C}^n/\Lambda$ be an $n$-dimensional complex torus. Let 
$$
\delta \colon H^1(T,\mathbb{Z}) \to H^2(Z,\mathbb{Z})
$$
be an homomorphism of $\mathbb{Z}$-modules such that $c_1(Z) \in \operatorname{Im} \delta$. We can identify $\delta$ with $c \in H^2(Z,\mathbb{Z}) \otimes \Lambda$, and hence it determines a topologically non-trivial principal $T$-bundle $\pi \colon M \to Z$ with characteristic class $c$. Assuming further that $\operatorname{Im} \delta \subset H^{1,1}(Z)$, $M$ can be endowed with a holomorphic structure. The first Chern class satisfies $c_1(M) = \pi^*c_1(Z) \in H^2(M,\mathbb{Z})$ and hence it vanishes because $c_1(Z) \in \operatorname{Im} \delta$ \cite{Hofer}. Furthermore, $M$ is non-K\"ahler by Blanchard's Theorem.

\begin{ex}\label{ex:TbundleRiemann}
Consider the case that $Z$ is a compact connected Riemann surface with genus $\geq 2$, and hence $c_1(Z) <0$. Let $\pi \colon M \to Z$ be a non-trivial principal $T$-bundle over $M$. By dimensional reasons the condition $c_1(Z) \in \operatorname{Im} \delta \subset H^{1,1}(Z)$ is always satisfied, and hence $c_1(M) = 0 \in H^2(M,\mathbb{Z})$. Choose a principal connection $\theta = (\theta_1,\ldots,\theta_{2n})$ on $M$ and define a $T$-invariant complex structure on $M$ by $J \theta_{2j-1} = \theta_{2j}$. Choose a K\"ahler metric $\omega_Z$ on $Z$ and consider the Hermitian form
$$
\omega = \pi^* \omega_Z + \sum_{j=1}^n \theta_{2j-1} \wedge \theta_{2j}.
$$
Then, we have
$$
dd^c \omega = -d^c \Bigg{(}\sum_{j=1}^n \pi^*F_{\theta_{2j-1}} \wedge \theta_{2j} - \theta_{2j-1} \wedge p^*F_{\theta_{2j}}\Bigg{)} = 0
$$
by dimensional reasons, where $F_{\theta_{j}}$ denotes the curvature of $\theta_j$. Therefore, $M$ is a pluriclosed manifold with vanishing first Chern class. Applying now Corollary \ref{cor:obstructionfib}, we conclude that $M$ does not admit a Bismut Hermitian-Einstein metric.
\end{ex}

\begin{ex} \label{ex:genT2bundle}
Let $Z$ be an algebraic complex surface with $c_1(Z) < 0$ and let $T = \mathbb{C}/\Lambda$. By the Aubin-Yau Theorem we have $\omega_Z$ a K\"ahler-Einstein metric on $Z$ with negative scalar curvature and $[\omega_Z] \in H^2(Z,\mathbb{Z})$. Choose $\alpha$ a primitive $(1,1)$-form with $[\alpha] \in H^2(Z,\mathbb{Z})$, and define $\delta$ so that its image is spanned by
$$
\delta_1 = - [\omega_Z] = c_1(Z), \qquad \delta_2 = [\alpha].
$$
On the corresponding $T$-bundle we can choose a connection $\theta = (\theta_1,\theta_2)$ with curvature 
$$
\frac{\i}{2\pi}F_\theta = (-\omega_Z,\alpha),
$$
and a $T$-invariant complex structure such that $J \theta_1 = \theta_2$. For any $u \in C^\infty(Z)$ we define the Hermitian metric $\omega = \pi^*e^u \omega_Z + \theta_1 \wedge \theta_2$. Then, a direct calculation shows that \cite{GGPoon}
$$
dd^c \omega = \pi^*(\Delta (e^u) \omega_Z^2 - \alpha^2 - \omega_Z^2),
$$
and hence the existence of a pluriclosed metric reduces to solve
$$
[\omega_Z]^2 = -  [\alpha] \cdot [\alpha] \in \mathbb{N}.
$$
Taking $\iota \colon Z \hookrightarrow \mathbb{P}^3$ a degree $d \geq 5$ projective hypersurface, we have $c_1(Z) = (4-d)\iota^*H$, for $H$ the hyperplane class, and we obtain the condition
$$
(d-4)^2 = - [\alpha] \cdot [\alpha] 
$$
for a primitive $(1,1)$-class $[\alpha]$. We have $H^2(Z,\mathbb{Z}) \cong \mathbb{Z}^{d(d(d-4) + 6) -2}$   and, assuming that $d$ is odd, the intersection pairing is given by the standard symmetric bilinear form with signature $-(d-2)d(d+2)/3$. Using Hirzebruch's formula for the Hodge numbers of projective hypersurfaces to calculate $h^{1,1}(M)$, one can prove that such a class always exists.  For example, taking $d = 5$ one has $H^2(Z,\mathbb{Z}) \cong \mathbb{Z}^{53}$, $h^{1,1}(M) = 45$, and signature $- 35$. Therefore, there exists a $36$-dimensional subspace of primitive $(1,1)$-classes, and hence there is a primitive $(1,1)$-class $[\alpha]$ with $[\alpha] \cdot [\alpha] = -1$.
\end{ex}

\section{Pluriclosed flow} \label{s:PCF}

In this section we turn to understanding the implications of Sections \ref{s:PCmetrics} and \ref{s:slopestability} for pluriclosed flow.  As explained in the introduction, a solution to pluriclosed flow is a one-parameter family of pairs $(\gw_t, \gb_t)$ on a complex manifold solving
\begin{align*}
\dt \gw =&\ - \rho_B^{1,1}, \qquad \dt \gb = - \rho_B^{2,0}.
\end{align*}
As a preliminary step we discuss an adaptation of the Schwarz Lemma to the setting of holomorphic Courant algebroids.  Given this, our first main task is to reinterpret pluriclosed flow as an evolution equation for the associated generalized Hermitian metrics $G_t$ (cf. Proposition \ref{p:Gflow}).  We use this to derive natural evolution equations for quantities measuring $G_t$ against some background choice of Hermitian metric on the associated holomorphic Courant algebroid, which is achieved using the Schwarz Lemma.  Next, in \S \ref{ss:PCFge} we give the proof of Theorem \ref{t:mainthm} using a priori estimates derived from the evolution equations for $G_t$ adapted to the case of a Bismut flat background.  Building on this we derive the consequences for generalized K\"ahler geometry in \S \ref{ss:GKflow}.  In \S \ref{ss:kod1} we prove Theorem \ref{t:kod1thm}.  The key point is to use that such surfaces are finitely covered by holomorphic principal $T^2$-bundles.  Using this structure we obtain partial estimates for $G_t$ by applying the Schwarz Lemma to the projection map of the fibration.  Using this partial estimate and then exploiting special properties of the Bismut curvature of a choice of invariant background metric yields the global existence.

\subsection{Schwarz Lemma}\label{s:Schwarz}

In this section we prove various useful formulae derived from the Schwarz Lemma for holomorphic sections on a holomorphic vector bundle. Let $\mathcal{W}$ be a holomorphic vector bundle over a complex manifold $(M,J)$. We fix a Hermitian metric $g$ on $M$ and define the Chern Laplacian on functions $f \in C^{\infty}(M)$ by 
$$
\Delta f := \frac{ n \i \partial \delb f  \wedge \omega^{n-1}}{\gw^n} = \tr_{\gw} (\i \del \delb f).
$$
Given a holomorphic section $w \in H^0(M,\mathcal{W})$, the classical Schwarz Lemma states that
\begin{equation}\label{eq:Schwarz}
\Delta |w|_h^2 = |\N^C_h w|^2_{g,h} - \IP{S^h_g w,w}_h
\end{equation}
for any choice of Hermitian metric $h$ on $\mathcal{W}$ with Chern connection $\N^C_h$. Note that $|\N^C_h w|^2_{g,h}$ is calculated using the background Hermitian metric $g$ jointly with the given Hermitian metric on the bundle. 

We are interested in the application of formula \eqref{eq:Schwarz} to the following general setup: Let $\pi_{\mathcal E}\colon \mathcal{E} \to M$ and $\pi_{\mathcal F} \colon \mathcal{F} \to N$ denote holomorphic vector bundles over complex manifolds $M$ and $N$ respectively.  Suppose $\Phi: \mathcal{E} \to \mathcal{F}$ is a morphism of holomorphic vector bundles covering $\phi : M\to N$, i.e. there is a commutative diagram
\[
\begin{tikzcd}
\mathcal E \arrow{r}{\Phi} \arrow[swap]{d}{\pi_{\mathcal E}} & \mathcal F \arrow{d}{\pi_{\mathcal F}} \\
M \arrow{r}{\phi}& N.
\end{tikzcd}
\]
Given such a map, there is a tautological holomorphic section of $\mathcal E^* \otimes \phi^* \mathcal F$ which, by abuse of notation, we denote also
$$
\Phi \in H^0(M, \mathcal E^* \otimes \phi^* \mathcal F).
$$
Furthermore, any pair of Hermitian metrics $G$ and $\til{G}$ on $\mathcal E$ and $\mathcal F$, respectively, induce a Hermitian metric $G^{-1}\otimes \til{G}$ on $\mathcal E^* \otimes \phi^* \mathcal F$. Observe that, for any choice of frames on $\mathcal E$ and $\mathcal F$, one has
$$
|\Phi|^2_{G^{-1},\tilde{G}} = (\phi^* \tilde G_{\alpha \overline{ \gamma}}) \Phi_j^\alpha \overline{\Phi_k^{\gamma}}G^{j \overline k}.
$$
Furthermore, the induced Chern connection $\nabla^{C,G^{-1},\til{G}}$ acts on $\Phi$, defining a tensor $A(G, \til{G}, \Phi)$ by
\begin{align} \label{f:Adef}
A(e) = (\nabla^{C,G^{-1},\til{G}}\Phi)(e) = \phi^*\nabla^{C,\til{G}}(\Phi e) - \Phi(\nabla^{C,G}e).
\end{align}
for any smooth section $e$ of $\mathcal E$. As a direct application of \eqref{eq:Schwarz} we obtain the following:

\begin{lemma} \label{l:Schwarz} Let $\mathcal{E} \to M$ and $\mathcal{F} \to N$ denote holomorphic vector bundles over complex manifolds $M$ and $N$ respectively, and suppose $\Phi: \mathcal{E} \to \mathcal{F}$ is a holomorphic map of vector bundles covering $\phi : M\to N$.  Given $g$ a Hermitian metric on $M$, $G$ a Hermitian metric on $\mathcal E$ and $\til{G}$ a metric on $\mathcal F$, one has
\begin{equation*}
\Delta|\Phi|^2_{G^{-1},\tilde{G}}= |A|^2_{g,G^{-1},\til{G}} + \IP{\Phi \circ S_g^{G} - \phi^*S_g^{\tilde G} \circ \Phi,\Phi}_{G^{-1},\tilde{G}}.
\end{equation*}
\end{lemma}

We next apply next Lemma \ref{l:Schwarz} to various situations of our interest. The simplest case is to apply it to the identity map of a fixed holomorphic vector bundle.  We note that in this case the tensor $A = \nabla^{C,G^{-1},\til{G}} \Phi$ is the difference of two Chern connections, and we use some more common notation for this:

\begin{defn} \label{d:Upsdef} Given $G, \til{G}$ Hermitian metrics on a holomorphic vector bundle $\mathcal{E}$ over $M$, let
\begin{align*}
\gU(G, \til{G}) := \N^C_G - \N^C_{\til{G}} \in T_{1,0}^* \otimes \End(\mathcal{E}).
\end{align*}
denote the difference of the associated Chern connections.  When taking the norm of $\gU$, we require a metric on $T_{1,0}^*$ as well as one on $\mathcal{E}$ and $\mathcal{E}^*$.  These choices will be denoted explicitly as subscripts, using possibly a given metric and its inverse on both $\mathcal{E}$ and $\mathcal{E}^*$ if it is not explicitly indicated.
\end{defn}

\begin{lemma}\label{l:Schwarz1} Let $\mathcal{E} \to M$ be a holomorphic vector bundle over a complex manifold $M$.  Given $g$ a Hermitian metric on $M$ and $G$ and $\til{G}$ Hermitian metrics on $\mathcal E$, one has that
\begin{equation*}
\Delta(\tr_{G}\tilde{G})= |\gU(G, \til{G})|^2_{g,G^{-1},\til{G}} + \tr_G \IP{\Big{(}S^{G}_g - S^{\tilde G}_g\Big{)}\cdot , \cdot}_{\tilde G}.
\end{equation*}
\begin{proof}
It follows from Lemma \ref{l:Schwarz}, setting $M = N$, $\mathcal{E} = \mathcal{F}$, and $\Phi = \Id$.
\end{proof}
\end{lemma}

We next consider specifically the case of generalized Hermitian metrics on exact holomorphic Courant algebroids, as in Proposition \ref{p:GQ0}.

\begin{lemma} \label{l:Schwarz2} 
Let $(M,g_0,J)$ be a pluriclosed manifold. Consider the associated orthogonal holomorphic vector bundle $\QQ = \mathcal Q_{\i \del \gw_0}$ as in Definition \ref{def:Q0}, and a generalized Hermitian metric $G = G(\gw, \gb)$ as in Proposition \ref{p:GQ0}. Choose an arbitrary Hermitian metric $\til{G}$ on $\QQ$. Then, one has that
\begin{equation*}
\Delta(\tr_{G}\tilde{G})= |\gU(G, \til{G})|^2_{g,G^{-1},\til{G}} + \tr_G \IP{\Big{(}S_g^{G} - S^{\til{G}}_g \Big{)}\cdot , \cdot}_{\tilde G}.
\end{equation*}
Furthermore, provided that we take $\til{G} = G(\omega_0,0)$, one also has that
\begin{align*}
\psi_{g_0}^{-1} \circ \gU(G, \til{G}) \circ \psi_{g} & = (\varphi_{g_0,g,\beta})\cdot \N^-_{g} - \N^-_{g_0}\\
\tr_{G}\tilde{G} & = \tr_{\tilde{G}} G = \tr_g g_0 + \tr_{g_0} g + |\beta|^2_{g,g_0}
\end{align*}
where $\psi_{g}$, $\psi_{g_0}$ are as in Section \ref{s:BismutId}, and $(\varphi_{g_0,g,\beta})\cdot\N^-_{g} $ denotes the action on $\N^-_{g}$ of the complex gauge transformation $\varphi_{g_0,g,\beta} \in \End (TM \otimes \mathbb{C})$, given by
$$
\varphi_{g_0,g,\beta}(X) = X^{1,0} + g_0^{-1}(gX^{0,1} + \sqrt{-1}\beta X^{1,0}) .
$$
\begin{proof}
The first part of the statement is a special case of Lemma \ref{l:Schwarz1}. Assuming now $\til{G} = G(\omega_0,0)$ and setting $G' = G(\omega,0)$, one has that
$$
\tr_{G}\tilde{G} = |e^{-\sqrt{-1}\beta}|^2_{G'^{-1},\tilde{G}} = (g_0)_{j \overline{k}} g^{j\overline{k}} + g_{j \overline{k}} g_0^{j\overline{k}} +  g_0^{l \overline{m}} g^{j\overline{k}}\beta_{jl} \overline{\beta_{km}} = \tr_g g_0 + \tr_{g_0} g + |\beta|^2_{g,g_0}.
$$
Finally, using Proposition \ref{p:GQ0}, one has that
\begin{align*}
\gU(G, \til{G}) & = (e^{ \sqrt{-1}\beta})^*\N^C_{G'} - \N^C_{\til{G}}\\
& = (e^{-\sqrt{-1}\beta} \circ \psi_g)_*\N^{-}_g - (\psi_{g_0})_*\N^{-}_{g_0}\\
& = \psi_{g_0} \circ ((\psi_{g_0}^{-1} \circ e^{-\sqrt{-1}\beta} \circ \psi_g)\cdot\N^{-}_g - \N^{-}_{g_0}) \circ \psi_{g_0}^{-1}
\end{align*}
and also that
\begin{align*}
 \psi_{g_0}^{-1} \circ e^{-\sqrt{-1}\beta} \circ \psi_g(X) & = \psi_{g_0}^{-1} \circ e^{- \sqrt{-1}\beta}(X^{1,0} - gX^{0,1})\\
 & =  \psi_{g_0}^{-1}(X^{1,0} - gX^{0,1} - \sqrt{-1}\beta X^{1,0})\\
 & = X^{1,0} + g_0^{-1}(gX^{0,1} + \sqrt{-1}\beta X^{1,0}).
\end{align*}
\end{proof}
\end{lemma}

\subsection{Pluriclosed flow and holomorphic Courant algebroids}

\begin{prop} \label{p:Gflow} Given $(M^{2n}, \gw_t, \gb_t, J)$ a solution to pluriclosed flow, the associated family of generalized Hermitian metrics $G_t$ on $\QQ_{\i \del \gw_0}$ satisfies
\begin{align*}
G^{-1} \dt G =&\ - S^G_g.
\end{align*}

\begin{proof}
Fix $p\in M$ and $t>0$. Conjugating by a constant $B$-field transformation we may assume without loss of generality that $\gb_t(p)=0$.  Given this setup, using Proposition \ref{p:GQ0} and Proposition \ref{p:SecondGChernCurvature} we can compute that 
\begin{align*}
\del_t G(X+ \xi,X+\xi) & = \del_t g (X,\overline X) - \del_t g (g^{-1}\xi,\overline{g^{-1}\xi}) +  \sqrt{-1} g^{-1}(\del_t \beta(X),\overline{\xi}) - \sqrt{-1} g^{-1}(\xi,\del_t \overline{\beta(X)})\\
& = - \sqrt{-1}  \del_t \omega (X,\overline X) - \sqrt{-1}  \del_t \omega (g^{-1}\xi,\overline{g^{-1}\xi})\\
& + \sqrt{-1} \del_t \beta(X,g^{-1}\overline{\xi}) - \sqrt{-1} \del_t \overline{\beta}(\overline{X},g^{-1}\xi),\\
G(S^G_g(X+ \xi),X+ \xi) & =  \sqrt{-1} G(-g^{-1}\rho_B^{1,1}X + g^{-1}\rho_B^{0,2}g^{-1}\xi + \rho_B^{2,0}X - \rho_B^{1,1} g^{-1} \xi,X+ \xi)\\
& =  \sqrt{-1} (-\rho_B^{1,1}(X,\overline{X}) + \rho_B^{0,2}(g^{-1}\xi,\overline{X})\\
& + \rho_B^{2,0}(X,g^{-1}\overline{\xi}) - \rho_B^{1,1}(g^{-1} \xi,g^{-1}\overline{\xi}))
\end{align*}
for any $X + \xi \in T^{1,0} \oplus T_{1,0}^*$, where we have used that $\del_t g = \del_t \omega(,J)$. Therefore, the family satisfies $G^{-1}\dt G = - S^G_g$ if and only if
\begin{align*}
\dt \gw =&\ - \rho_B^{1,1}, \qquad \dt \gb = - \rho_B^{2,0},\qquad \dt \overline{\gb} = - \rho_B^{0,2}.
\end{align*}
Using now that $\rho_B$ is a real two-form, the statement follows.
\end{proof}
\end{prop}

The Schwarz Lemma computations from Section \ref{s:Schwarz} will help us derive useful evolution equations for measuring the metric along pluriclosed flow.  

\begin{prop} \label{p:paraSL} Fix $(M^{2n}, \gw_t, \gb_t, J)$ a solution to pluriclosed flow, with $G_t$ the associated family of generalized Hermitian metrics on $\QQ_{\i \del \gw_0}$.  Given $\til{G}$ a Hermitian metric on $\QQ_{\i \del \gw_0}$, we have
\begin{align*}
\left(\dt - \gD \right) \tr_G \til{G} =&\ - \brs{\gU(\til{G},G)}^2_{g,G^{-1},\til{G}} + \tr_G \IP{ S^{\til{G}}_g \cdot, \cdot}_{\til{G}}.
\end{align*}
\begin{proof} 
Applying Lemma \ref{l:Schwarz2} we obtain
\begin{align*}
\dt \tr_G \til{G} =&\ \tr G^{-1} S^G_g \til{G}\\
=&\ \tr_G \IP{ S^{{G}}_g \cdot, \cdot}_{\til{G}}\\
=&\ \Delta \tr_{G} \til{G} - \brs{\gU(\til{G},G)}^2_{g,G^{-1},\til{G}} + \tr_G \IP{ S^{\til{G}}_g \cdot, \cdot}_{\til{G}},
\end{align*}
as required.
\end{proof}
\end{prop}

\begin{prop} \label{p:Upsev} Fix $(M^{2n}, \gw_t, \gb_t, J)$ a solution to pluriclosed flow, with $G_t$ the associated family of generalized metrics on $\QQ_{\i \del \gw}$.  Given $\til{G}$ a Hermitian metric on $\QQ_{\i \del \gw}$, we have
\begin{align*}
\left( \dt -\gD \right) \brs{\gU(G, \til{G})}^2_{g,G^{-1},G} =&\ - \brs{\N \gU}^2_{g,G^{-1},G} - \brs{\bar{\N} \gU + \bar{T} \cdot \gU}^2_{g,G^{-1},G}\\
&\ + T \star \gU \star \Omega^{\til{G}} + \gU \star \bar{\gU} \star \Omega^{\til{G}} + \gU \star \N^{\til{G}} \Omega^{\til{G}},
\end{align*}
where
\begin{align*}
\left(\bar{T} \cdot \gU \right)_{\bi j A}^B =&\ g^{\bk l} \bar{T}_{\bi \bk  j} \gU_{l A}^B.
\end{align*}
\begin{proof} As we have established the evolution equation $G^{-1} \dt G = -S^G_g$ in the general setting here of twisted Courant algebroids in Proposition \ref{p:Gflow}, the result follows formally exactly as in \cite[Proposition 3.3]{JordanStreets}.
\end{proof}
\end{prop}

\subsection{Global existence and convergence on Bismut-flat backgrounds} \label{ss:PCFge}

In this subsection we will finish the proof of Theorem \ref{t:mainthm}.  The proof relies principally on the evolution equations for $G$ derived above.  Using the assumption of a background metric which is Bismut-flat, it is possible to remove the reaction terms in the heat equations of Propositions \ref{p:paraSL} and \ref{p:Upsev}, leading directly to a priori estimates on $G$.  These will then lead to full regularity and convergence of the pluriclosed flow.  Before giving the proof we prove a technical lemma which clarifies that a priori estimates for $G$ translate in a natural way to various a priori estimates for the associated pluriclosed metric $\gw$ and torsion potential $\gb$.

\begin{lemma} \label{l:gentoclassical} Suppose $(M^{2n}, J)$ is a complex manifold, and suppose $(\gw, \gb)$ and $(\til{\gw}, \til{\gb})$ are pluriclosed metrics on $M$ such that the associated generalized Hermitian metrics $G$ and $\til{G}$ are defined on the same holomorphic Courant algebroid $\QQ$ and satisfy
\begin{align*}
\Lambda^{-1}\tilde{G} \leq G\leq \Lambda \tilde{G}, \qquad \brs{\gU(G,\tilde{G})}_{g,G}<\gL.
\end{align*}
Then there exists a constant $A = A(n, \gL)$ such that
\begin{align*}
A^{-1} \til{g} \leq g \leq A \til{g}, \qquad \brs{\gb}_{\til{g}} \leq A, \qquad \brs{\gU(g, \til{g})}_{\til{g}} \leq A.
\end{align*}
	
	\begin{proof}
	By the assumed uniform equivalence of $G$ and $\tilde{G}$ and their explicit expressions from Proposition \ref{p:GQ0}, it follows that for $\xi \in \Lambda^{1,0}$ we have
	$$\Lambda^{-1}\tilde{g}^{-1}(\xi,\bar{\xi}) = \Lambda^{-1}\tilde{G}(\xi,\bar{\xi})\leq G(\xi,\bar{\xi}) = g^{-1}(\xi,\bar{\xi}) \leq \Lambda\tilde{G}(\xi,\bar{\xi}) = \Lambda\tilde{g}^{-1}(\xi,\bar{\xi}).$$ 
	This implies the claimed uniform equivalence of $g$ and $\til{g}$.  A similar argument using sections of the tangent bundle then yields the upper bound for $\brs{\gb}_{\til{g}}$.
	
	To estimate the connection, we first recall the computation of the connection coefficients in Lemma \ref{l:GChernCxn}
	$$(\gU^G)_{iW^a}^{Z^b} = g^{\bc b}g^{\bq a}\bar{T}_{\bq\bc i} - \tilde{g}^{\bc b}\tilde{g}^{\bq a}\bar{\tilde{T}}_{\bq\bc i}.$$
	Taking norms and using the uniform equivalence of $g$ and $\til{g}$ and the estimate for $\gU(G, \til{G})$ we obtain
	$$|T|_{\til{g}} \leq C.$$ 	
	Now turning to the tangent-tangent component of $\gU^G$ from Lemma \ref{l:GChernCxn} we can express
	$$(\gU^g)_{ia}^b = (\gU^G)_{i Z^a}^{Z^b} - \sqrt{-1} g^{\bc b}g^{\bq p}\gb_{ap}\bar{T}_{\bq\bc i} +  \sqrt{-1} \tilde{g}^{\bc b}\tilde{g}^{\bq p}\tilde{\gb}_{ap}\bar{\tilde{T}}_{\bq\bc i}.$$
	We have estimated all terms on the right hand side of this equation, thus the estimate for $\gU(g, \til{g})$ follows.
	\end{proof}
\end{lemma}

\begin{proof}[Proof of Theorem \ref{t:mainthm}] Let $\gw_F$ denote the given Bismut-flat metric, and let $\mathcal Q_{\i\del \gw_F}$ denote the holomorphic Courant algebroid associated to $[\del \gw_{F}]$.  Furthermore, let $G_F$ denote the Hermitian metric on $\QQ_{\i \del \gw_F}$ associated to $\gw_F$ via Proposition \ref{p:GQ0}.  Now given $\gw_0$ another pluriclosed metric satisfying $[\del \gw_0] = [\del \gw_F] \in H^{2,1}_{\delb}$, we can choose $\gb \in \Lambda^{2,0}$ such that
\begin{align*}
\delb {\gb} = \del \gw_F - \del \gw_0.
\end{align*}
Let $G_0$ denote the metric associated to $(\gw_0, \gb_0)$ as in Proposition \ref{p:GQ0}.  By \cite{PCF}, there exists $\ge > 0$ and a solution $G_t$  to pluriclosed flow with initial data $G_0$ on $[0, \ge)$.  Since $\gw_F$ is Bismut-flat, it follows from Proposition \ref{p:Bismutid} that the Chern curvature of $G_F$ vanishes, and thus we obtain from Proposition \ref{p:paraSL} the evolution equations
\begin{gather} \label{f:mt10}
\begin{split}
\left(\dt - \gD \right) \tr_{G} G_F =&\ -  \brs{\gU(G_F,G)}^2_{g,G^{-1},G_F}.
\end{split}
\end{gather}
It follows from the maximum principle that, for any interval $[0,T]$ on which the solution exists,
\begin{align*}
\sup_{M \times [0,T]} \tr_{G} G_F \leq&\ \sup_{M \times \{0\}} \tr_{G} G_F.
\end{align*}
We note that for two generalized Hermitian metrics $G$, $\til{G}$ on a fixed holomorphic Courant algebroid it follows that $\tr_G \til{G} = \tr_{\til{G}} G$ (see Lemma \ref{l:Schwarz2}).  Thus there exists a uniform constant $\gL > 0$ so that for any time $t$ one has
\begin{align} \label{f:mt15}
\gL^{-1} G_F \leq G_t \leq \gL G_F.
\end{align}
Furthermore, let $\gU = \gU(G_t, G_F)$ as in Definition \ref{d:Upsdef}.  Again using that the Chern curvature of $G_F$ vanishes, It follows from Proposition \ref{p:Upsev} that
\begin{align} \label{f:mt20}
\left(\dt - \gD \right) \brs{\gU}^2_{g,G} =&\ - \brs{\N \gU}^2_{g,G} - \brs{\bar{\N} \gU + \bar{T} \cdot \gU}^2_{g,G}.
\end{align}
It follows from the maximum principle that for any interval $[0,T]$ on which the solution exists,
\begin{align*}
\sup_{M \times [0,T]} \brs{\gU(G, G_F)}^2_{g,G} \leq \sup_{M \times \{0\}} \brs{\gU(G, G_F)}^2_{g,G}.
\end{align*}
Using Lemma \ref{l:gentoclassical} we thus obtain uniform equivalence and a $C^1$ bound for the classical objects $(\gw_t, \gb_t)$.  We can now argue as in \cite[Theorem 1.2]{JordanStreets}, there are uniform $C^{\infty}$ estimates for $G_t$ and $g_t$ for all times.  Thus the flow exists for all time, finishing the claim of long-time existence.

To show convergence we first note that by putting together (\ref{f:mt10}) and (\ref{f:mt20}), and using the uniform equivalence estimate of (\ref{f:mt15}), it follows that there is a constant $A > 0$ so that
\begin{align*}
\left(\dt - \gD \right) \left( t \brs{\gU(G,G_F)}^2_{g,G} + A \tr_{G_F} G \right) \leq 0.
\end{align*}
It follows from the maximum principle that for any time $t > 0$ one has
\begin{align} \label{f:mt30}
\sup_{M \times \{t\}} \brs{\gU(G, G_F)}^2_{g,G} \leq \frac{A \sup_{M \times \{0\}} \tr_{G_F} G}{t}.
\end{align}
Using the uniform $C^{\infty}$ estimates for $G_t$, every sequence of times $\{t_j\} \to \infty$ contains a subsequence which converges to a limiting metric $G_{\infty}$.  By (\ref{f:mt30}) it follows that $\gU(G_{\infty}, G_F) = 0$, and thus $G_{\infty}$ is Chern-flat.  Choosing such a flat limit $G_{\infty}$, we can repeat the above analysis with $G_F$ replaced by $G_{\infty}$.  In particular, for a large time $t$ such that $\brs{G_{t} - G_{\infty}}_{G_{\infty}} \leq \ge$, we have
\begin{align*}
2n \leq \sup_{M \times \{t\}} \tr_{G_{\infty}} G \leq 2n + C \ge
\end{align*}
for some uniform constant $C$.  These inequalities will be preserved for all times larger than $t$ by (\ref{f:mt10}), and then the convergence to $G_{\infty}$ follows.
\end{proof}

\subsection{Contractibility of the space of generalized K\"ahler structures} \label{ss:GKflow}

In this subsection we prove Corollary \ref{c:GKcor}, regarding generalized K\"ahler structures on Bismut flat manifolds. Key to our argument is the Poisson tensor 
$$
\gs = \tfrac{1}{2} g^{-1} [I,J]
$$ 
associated to a generalized K\"ahler structure $(g, I, J)$. The tensor $\gs$ was discovered by Pontecorvo-Hitchin \cite{AGG,PontecorvoCS, HitchinPoisson}, and was shown to be constant along a solution to generalized K\"ahler-Ricci flow in \cite[Corollary 1.5]{GibsonStreets}.

\begin{proof}[Proof of Corollary \ref{c:GKcor}] Let $(M^{2n}, \gw_F, J)$ be a Bismut-flat manifold, and fix $(g, I, J)$ a generalized K\"ahler structure on $M$ such that $[\del \gw_J] = [\del \gw_F] \in H^{2,1}_{\delb}$.  By Theorem \ref{t:mainthm}, the solution $(g_t, I_t, J)$ to generalized K\"ahler-Ricci flow with initial condition $(g, I, J)$ exists for all time, and $g_t$ converges to a Bismut-flat metric.  Using \cite[Lemma 9.27]{garciafernStreets2020} we can express the evolution equation for $I$ as
\begin{align*}
\dt I =&\ L_{\theta_I^{\sharp} - \theta_J^{\sharp}} I = \rho_B^J \cdot \gs.
\end{align*}
As the evolution of $G_t$ is determined by $\rho_B$ and $G_t$ is converging to a limit $G_{\infty}$, it follows that $I$ will also converge to a smooth limit $I_{\infty}$ as claimed.  As $I_t = \phi_t^* I$ and $I_t$ has uniform $C^{\infty}$ estimates, it follows that the diffeomorphisms $\phi_t$ also converge to a limit $\phi_{\infty}$, and thus $I_{\infty}$ is biholomorphic to $I$, as claimed.
\end{proof}

\subsection{Global existence on complex surfaces with \texorpdfstring{$\gk \geq 0$}{k >= 0}} \label{ss:kod1}

In this subsection we prove Theorem \ref{t:kod1thm}.  A key point is to find a background metric with certain curvature properties for every choice of $\QQ$.
To begin, we show that every class in $H^{1,1}_{BC}$ is represented by an invariant form. We build on an observation of Teleman, which gives an explicit characterization of the failure of the $\del\delb$-lemma on compact complex surfaces.  

\begin{lemma}[\cite{Telemancone}]\label{l:Teleman}

Let $(M^4, J)$ be a compact complex surface, and let
\begin{align*}
B^{1,1}_{\mathbb R} = \{ \mu \in \Lambda^{1,1}_{\mathbb R}\ |\ \exists a \in \Lambda^1_{\mathbb R} = d a \}.
\end{align*}
Then there exists an exact sequence
\begin{align*}
0 \to \i \del \delb \Lambda^0_{\mathbb R} \to B^{1,1}_{\mathbb R} \to \mathbb R,
\end{align*}
where the final map is the $L^2$ inner product with a pluriclosed metric.
\end{lemma}

\begin{lemma} \label{l:invmetrics} Let $(M^4, J)$ be a compact complex surface which is the total space of a holomorphic principal $T^2$-bundle.  Given $\gw$ a pluriclosed metric on $M$ there exists a $T^2$-invariant metric in $[\gw] \in H^{1,1}_A$.
\begin{proof} 
First choose $\hat{\gw}$ a $T^2$-invariant pluriclosed metric on $M$, which always exists by averaging an arbitrary pluriclosed metric over the $T^2$-action.  We first use this to show that every class in $H^{1,1}_{BC}$ admits $T^2$-invariant representatives.  Now fix $[\phi] \in H^{1,1}_{BC}$.  We can define
\begin{align*}
\hat{\phi} := \int_{g \in T^2} g^* \phi.
\end{align*}
The form $\hat{\phi}$ is $T^2$-invariant, and to show that $[\hat{\phi}] = [\phi] \in H^{1,1}_{BC}$, it suffices by a standard chain-homotopy argument to show that for $X$ any vector field tangent to the $T^2$-action, one has $L_X \phi \in \i \del \delb \Lambda^0_{\mathbb R}$.  By Lemma \ref{l:Teleman}, it suffices to show that the $L^2$ inner product with $\hat{\gw}$ vanishes.  Using Stokes Theorem and the Cartan formula we compute
\begin{align*}
\IP{L_X \phi, \hat{\gw}}_{\hat{\gw}} = \int_M L_X \phi \wedge \hat{\gw} = \int_M L_X \left( \phi \wedge \hat{\gw} \right) = \int_M d i_X \left( \phi \wedge \hat{\gw} \right) = 0,
\end{align*}
as required.  Now knowing this, we assume that $\hat{\gw}$ is the average of $\gw$ over the $T^2$-action, and show that $[\hat{\gw}] = [\gw] \in H^{1,1}_A$.  It suffices to prove that the infinitesimal action preserves Aeppli cohomology classes, and for this we use that integration gives a perfect pairing between $H_{A}^{1,1}$ and $H_{BC}^{1,1}$.  Thus we fix $[\phi] \in H^{1,1}_{BC}$ with a $T^2$-invariant representative $\hat{\phi}$, and integrate by parts to conclude 
\begin{align*}
\int_M L_X \gw \wedge \hat{\phi} = - \int_M \gw \wedge d (i_X \hat{\phi}) = - \int_M \gw \wedge L_X \hat{\phi} = 0,
\end{align*}
as required.
\end{proof}
\end{lemma}

Next we record a key lemma computing the Bismut curvature tensor of a $T^2$-invariant pluriclosed metric.  Such invariant metrics are described by the Kaluza-Klein ansatz (cf. \cite[Definition 5.1]{Streetssolitons}), and the curvature computation below is implicit in \cite[Proposition 5.12]{Streetssolitons}.

\begin{lemma} \label{l:invcurvature} Let $(M^4, J)$ be a compact complex surface which is the total space of a holomorphic principal $T^2$-bundle over a Riemann surface $\Sigma$.  Let $\gw$ denote a $T^2$-invariant pluriclosed metric on $M$, expressed as
\begin{align*}
\gw = \pi^* \gw_{\Sigma} + \tr_h \mu \wedge J \mu,
\end{align*}
where $\gw_{\Sigma}$ is a metric on $\Sigma$,  $\mu + J \mu$ is a Hermitian connection, and $h$ is an inner product on $\mathfrak t^2$.  Then
\begin{align*}
\Omega^B =&\ \tfrac{1}{2} \left( R_{\gw_{\Sigma}} - \brs{F_{\mu}}^2_{g_{\Sigma}, h} \right) \pi^* \gw_{\Sigma} \otimes \pi^* \gw_{\Sigma} + h \left( d \tr_{\gw_{\Sigma}} F_{\mu}, \cdot \right) \otimes \pi^* \gw_{\Sigma}.
\end{align*}
\end{lemma}

\begin{proof}[Proof of Theorem \ref{t:kod1thm}] Fix $(M^4, J)$ a minimal compact complex non-K\"ahler surface of Kodaira dimension $\gk \geq 0$.  It follows from the Kodaira classification of surfaces (cf. \cite{Wall} \S 7) that $M$ must be an elliptic fibration, with only multiple fibers.  In particular, it follows that there exists a finite cover of $M$ which admits a holomorphic principal $T^2$ action, and it suffices to show global existence on such manifolds.  In particular we suppose $\pi : M^4 \to \Sigma$ is a holomorphic $T^2$-bundle over a compact Riemann surface $\Sigma$ where $\chi(\Sigma) < 0$ if $\gk = 1$ and $\chi(\Sigma) = 0$ if $\gk = 0$.

Fix $\gw_0$ a pluriclosed metric on $M$.  By Lemma \ref{l:invmetrics}, there exists a $T^2$-invariant metric $\hat{\gw} = \pi^* \gw_{\Sigma} + \tr_h \mu \wedge J \mu \in [\gw_0]$.  We can modify the metric on $\Sigma$ so that $R_{\Sigma}$ is constant, and by Hodge theory further modify the principal connection $\mu$ to assume that $\tr_{\gw_{\Sigma}} F_{\mu}$ is constant.  These changes preserve the associated Aeppli cohomology class on $M$ and so we assume without loss of generality that $\hat{\gw}$ satisfies these conditions.
The metric $\hat{\gw}$ defines a holomorphic Courant algebroid $\QQ_{\i \del \hat{\gw}}$ together with a generalized Hermitian metric $\hat{G}$.  Now by construction we can choose $\gb \in \Lambda^{2,0}$ such that
\begin{align*}
\delb {\gb} =  \del \hat{\gw} - \del \gw_0 .
\end{align*}
Let $G_0$ denote the metric associated to $(\gw_0, \gb_0)$ as in Proposition \ref{p:GQ0}.  By \cite{PCF}, there exists $\ge > 0$ and a solution $G_t$  to pluriclosed flow with initial data $G_0$ on $[0, \ge)$.

We first obtain a partial estimate on the metric using the fibration structure and the Schwarz Lemma.  The holomorphic Courant algebroid $\QQ$ comes equipped with a natural holomorphic projection map onto $T^{1,0}_M$ which we denote $\pi_{\QQ}$.  We furthermore obtain from the fibration structure the holomorphic map $d \pi : T^{1,0}_M \to T^{1,0}_{\Sigma}$.  Composing these yields the holomorphic map of vector bundles $\Phi = d \pi \circ \pi_{\QQ} : \QQ \to T^{1,0}_{\Sigma}$.  It follows from the construction and Proposition \ref{p:GQ0} that
\begin{align*}
\brs{\Phi}^2_{G, g_{\Sigma}} = \tr_{\gw} \pi^* \gw_{\Sigma}.
\end{align*}
Furthermore, using Lemma \ref{l:Schwarz} (where $A = A(G, g_{\Sigma}, \Phi)$ is defined by (\ref{f:Adef})) we obtain
\begin{align*}
\dt \tr_{\gw} \pi^* \gw_{\Sigma} =&\ \dt \brs{\Phi}^2_{G,g_{\Sigma}}\\
=&\ \IP{\Phi \circ S^G_g, \Phi}_{G^{-1}, g_{\Sigma}}\\
=&\ \gD_g \brs{\Phi}^2_{G^{-1}, g_{\Sigma}}  - \brs{A}^2_{g,G^{-1}, g_{\Sigma}} + \IP{ S^{g_{\Sigma}}_g \circ \Phi, \Phi}_{G^{-1}, g_{\Sigma}}\\
=&\ \gD_g \tr_{\gw} \pi^* \gw_{\Sigma} - \brs{A}^2_{g,G^{-1}, g_{\Sigma}} + \tfrac{1}{2} R_{\Sigma} \left( \tr_{\gw} \pi^* \gw_{\Sigma} \right)^2.
\end{align*}
Note that by construction $R_{\Sigma}$ is constant, either $-2$ or $0$ depending on whether $\gk = 1$ or $0$.
By the maximum principle we conclude for any smooth existence time $T > 0$ the estimate
\begin{align} \label{f:kod20}
\sup_{M \times \{T\}} \tr_{\gw} \pi^* \gw_{\Sigma} \leq \left(C + \tfrac{1}{2} \brs{R_{\Sigma}} T \right)^{-1}.
\end{align}

We next establish the uniform equivalence of the metrics $G_t$ along the flow.  Combining Proposition \ref{p:paraSL} with Proposition \ref{p:Bismutid}, the curvature computation of Lemma \ref{l:invcurvature}, and the estimate (\ref{f:kod20}) we obtain for a topological constant $\gl$,
\begin{align*}
\left(\dt - \gD \right) \tr_G \hat{G} =&\ - \brs{\gU(G, \hat{G})}^2_{g,G^{-1}, \hat{G}} + \tr_G \IP{ S^{\til{G}}_g \cdot, \cdot}_{\til{G}}\\
=&\ - \brs{\gU(G, \hat{G})}^2_{g,G^{-1}, \hat{G}} + \tr_G \IP{ \tr_g \left( \gl \pi^* \gw_{\Sigma} \otimes \psi_* \pi^* \gw_{\Sigma} \right) \cdot , \cdot}_{\hat{G}}\\
\leq&\ C \left( \tr_g \pi^* \gw_{\Sigma} \right) \tr_{G} \hat{G}\\
\leq&\ C \tr_{G} \hat{G}.
\end{align*}
By the maximum principle we conclude
\begin{align*}
\sup_{M \times \{T\}} \tr_G \hat{G} \leq e^{CT}.
\end{align*}
This implies that $G_t$ and $\hat{G}$ are uniformly equivalent on any compact time interval, and from Proposition \ref{p:Upsev} we conclude
\begin{align*}
\left(\dt - \gD \right) \brs{\gU(G, \hat{G})}^2_{g,G^{-1},G} \leq&\ C \left(1 + \brs{\gU(G, \hat{G})}^2_{g,G^{-1},G} \right).
\end{align*}
By the maximum principle we obtain a uniform estimate for $\brs{\gU(G, \hat{G})}^2_{g,G^{-1},G}$ on any finite time interval, and the proof of long-time existence now concludes as in Theorem \ref{t:mainthm}.
\end{proof}

\begin{rmk} \label{r:convrmk} For $\gw_t$ a pluriclosed flow as in Theorem \ref{t:kod1thm} one expects the blowdown limits $\frac{\gw_t}{2t}$ to converge with bounded curvature to either a point in the case $\gk = 0$, or to $(\Sigma, \gw_{\Sigma})$ in the case $\gk = 1$, where $\gw_{\Sigma}$ denotes the unique conformal metric of curvature $-1$.  A basic fact in this direction is that pluriclosed flow preserves the area of the $T^2$ fibers, and thus the area of the fibers goes to zero along any blowdown sequence.  Furthermore, assuming the solution satisfies type III curvature and diameter bounds, this kind of limiting behavior can be derived by the use of an expanding entropy functional \cite{GindiStreets}.  Assuming the initial metric is $T^2$-invariant this behavior was shown in \cite{StreetsRYMPCF}.
\end{rmk}


\end{document}